\documentclass[11pt]{article}
\usepackage[utf8]{inputenc}
\usepackage[margin=1in]{geometry}
\usepackage{amsmath}
\usepackage{amsthm}
\usepackage{amssymb}

\newtheorem{theo}{Theorem}[section]
\newtheorem{definition}{Definition}[section]
\newtheorem{prop}[theo]{Proposition}
\newtheorem{lemma}[theo]{Lemma}

\newtheorem{conj}[theo]{Conjecture}

\newtheorem{claim}[theo]{Claim}

\newcommand{\HH}{{\cal H}}

\newcommand{\CC}{{\cal C}}

\newcommand{\eps}{{\varepsilon}}
\begin{document}
\date{}

\title{
Problems and results in Extremal Combinatorics - IV
}

\author{Noga Alon
\thanks
{Department of Mathematics, Princeton University,
Princeton, NJ 08544, USA and
Schools of Mathematics and
Computer Science, Tel Aviv University, Tel Aviv 69978,
Israel.  
Email: {\tt nogaa@tau.ac.il}.  
Research supported in part by
NSF grant DMS-1855464, ISF grant 281/17, BSF grant 2018267
and the Simons Foundation.}
}

\maketitle
\begin{abstract}
Extremal Combinatorics is among the most active topics in
Discrete Mathematics, dealing with problems that are often
motivated by questions in other areas,
including Theoretical Computer Science and Information Theory.
This paper contains a
collection of problems and results in the area, including solutions
or
partial solutions to open problems suggested by various
researchers.
The topics considered here include questions in
Extremal Graph Theory, Coding Theory and Social Choice.
This is by no means a comprehensive survey of the area,
and is merely a collection of
problems, results and proofs, which are hopefully
interesting.  
As the title of the paper suggests, this is a sequel
of three previous paper \cite{Al1}, \cite{Al2}, \cite{Al3}
of the same flavour.
Each section of this paper is essentially
self contained, and can be read separately.
\end{abstract}

\section{Maintaining high girth in graph packings}

We say that two $n$-vertex graphs $G_1$ and $G_2$
{\em pack} if
there exists an edge-disjoint
placement of them on the same set of $n$ vertices. 
There is an extensive literature
dealing with sufficient conditions ensuring that two graphs
$G_1$ and $G_2$ on $n$ vertices pack. A well known open conjecture on the
subject is the one of Bollob\'as and Eldridge
\cite{BE} asserting that if the maximum degrees in $G_1$ and
$G_2$ are $d_1$ and $d_2$, respectively, and if
$(d_1+1)(d_2+1) \leq n+1$ then $G_1$ and $G_2$ pack.
Sauer and Spencer (\cite{SS}, see also Catlin \cite{Ca}),
proved that this is the case if
$2 d_1 d_2 <n$. For a survey of packing results including
extensions, variants and relevant references, see
\cite{KKY}.

A natural extension of the packing problem is that of requiring
a packing in which the girth of the combined  graph whose edges
are those of the two packed graphs is large, assuming this is the
case for each of the individual graphs. Indeed, in the basic problem
the girth of each of the 
packed graphs exceeds $2$, and the packing condition
is simply the requirement that in the combined graph the girth 
exceeds $2$. Here we prove such an extension, observe that it implies 
the old result of Erd\H{o}s and Sachs \cite{ES} about the existence of high-
girth regular graphs, and describe an application for obtaining an
explicit construction of high-girth directed expanders.
\begin{theo}
\label{t111}
Let $G_1=(V_1,E_1)$ and $G_2=(V_2,E_2)$ 
be two $n$-vertex graphs, let $d_1$ be the
maximum degree of $G_1$ and let $d_2$ be the maximum degree of
$G_2$. Suppose the girth of each of the graphs $G_i$ is at least
$g >2$ and let $k$ be the largest integer satisfying
\begin{equation}
\label{e111}
1+(d_1+d_2)+(d_1+d_2)(d_1+d_2-1)+
\ldots + (d_1+d_2)(d_1+d_2-1)^{k-1} < n
\end{equation}
Then there is a packing of the two graphs so that the combined
graph has girth at least $\min\{g, k\}$.
\end{theo}
Note that for fixed $d_1+d_2 \geq 3$ and large $n$, the number $k$
above is $(1+o(1))\frac{\log n}{\log (d_1+d_2-1)}$.

\subsection{Proof}
Clearly we may assume that both $G_1$ and $G_2$ have  edges, thus
$d_1$ and $d_2$ are positive. If $2d_1d_2 \geq n$ 
then $d_1+d_2 \geq \sqrt {2n}$ implying that
$1+(d_1+d_2)+(d_1+d_2)(d_1+d_2-1) \geq
1 +  \sqrt {2n}+ \sqrt {2n} (\sqrt {2n} -1) =2n+1>n$, that is, the
largest $k$ satisfying (\ref{e111}) is at most $1$. In this case the
conclusion is trivial since $\min\{g,k\} \leq 1$ and any
placement of $G_1,G_2$ will do. We thus may and will assume that
$2d_1d_2<n$. Since any placement will do even if $k=2$ we assume
that $k \geq 3$.
By the result of \cite{SS}, $G_1,G_2$ pack.
Among all possible packings choose one in which the girth $m$ of the
combined graph is maximum, and the number of cycles of length $m$
in this combined graph is minimum (if the girth is infinite there
is nothing to prove). Suppose this packing is given by two
bijections $f_1:V_1 \mapsto V$ and $f_2:V_2 \mapsto V$
where $V$ is the fixed set of $n$  vertices of the
combined graph, which we denote by $H=(V,E)$. As $G_1$ and $G_2$ pack,
$m \geq 3$.
ref{e111}). Let
$v_1=f_1^{-1}(v)$ be the preimage of $v$ in $V_1$. Let 
$f_1': V_1 \mapsto V$ be the bijection obtained from $f_1$ by
swapping the images of $u_1$ and $v_1$. Formally, $f'_1(u_1)=v,
f'_1(v_1)=u$, and $f'_1(w)=f_1(w)$ for all $w \in V_1-\{u_1,v_1\}$. 

We claim that in the embedding of $G_1,G_2$ given by $f'_1,f_2$
the girth of the combined graph, call it $H'$, is at least $m$  and
the number of cycles of length $m$ in $H'$ is smaller than the
corresponding number in $H$, contradicting the minimality in the
choice of $f_1,f_2$. To prove this claim put $u_2=f_2^{-1}(u)$,
$v_2=f_2^{-1}(v)$.
Let $X_1$ denote the set of
images under $f_1$ of all the neighbors of $u_1$ in $G_1$, and let
$X_2$ denote the set of images under $f_2$ of all neighbors of
$u_2$ in $G_2$. Similarly, let
$Y_1$ be the set of
images under $f_1$ of all the neighbors of $v_1$ in $G_1$, and let
$Y_2$ be the set of images under $f_2$ of all neighbors of
$v_2$ in $G_2$. Note that since $m \geq 3$ and
$k+1 \geq 3$ all four sets $X_1,X_2,Y_1,Y_2$ are pairwise disjoint.
The cycles of length $m$ in $H$ and $H'$ that do
not contain any of the two vertices $u,v$ are exactly the same
cycles. On the other hand, the 
cycle $C$ is of length $m$ and it exists in $H$ but not in $H'$, since
all edges of $H$ between $u$ and $X_1$ do not belong to
$H'$, and $C$ contains such an edge (as well as an edge from
$u$ to $X_2$). Any cycle $C'$ of $H'$ that is not a cycle of $H$ must
contain at least one edge either between $u$ and $Y_1$ or between
$v$ and $X_1$ (or both).  Consider the following possible cases.
\vspace{0.2cm}

\noindent
{\bf Case 1a:}\, $C'$ contains $u$ but not $v$ and contains two
edges from $u$ to $Y_1$. In this case the cycle of $H$ obtained 
from $C'$ by replacing $u$ by $v$ is of the same length as $C'$.
This is a one-to-one correspondence between cycles as above of length
$m$ in $H'$ and in $H$ (if there are any such cycles).
\vspace{0.2cm}

\noindent
{\bf Case 1b:}\, $C'$ contains $u$ but not $v$ and contains an edge
$uy_1$ from $u$ to $Y_1$ and an edge $ux_2$ from $u$ to $X_2$. 
In this case the part of the cycle between $y_1$ and $x_2$ which
does not contain $u$ is a path in $H$ between $y_1$ and $x_2$. The
length of this path is at least $k-1$, since the distance in $H$
between $u$ and $v$ is at least $k+1$. Therefore, the length of
$C'$ is at least $(k-1)+2=k+1>m$.
\vspace{0.2cm}

\noindent
{\bf Case 1c:}\, $C'$ contains $v$ but not $u$: this is symmetric
to either Case 1a or Case 1b.
\vspace{0.2cm}

\noindent
{\bf Case 2a:}\, $C'$ contains both $u$ and $v$ and contains two
edges $uy_1,uy'_1$ from $u$ to $Y_1$. If both neighbors of $v$
in $C'$ belong to $Y_2$ then each of the parts of $C'$ connecting any of
them to $y_1$ or to $y'_1$ is of length at least
$m-2$, since the girth of $H$ is $m$, hence the total length
of $C'$ is at least $2(m-2)+4=2m>m$. If both neighbors of
$v$ in $C'$ are in $X_1$ then since the distance in $H$ between 
$X_1$ and $Y_1$ is at least $k-1$, in this case the length of $C'$
is at least $2(k-1)+4=2k+2>m$. If the two neighbors of $v$
in $C'$ are $y_2 \in Y_2$ and $x_1 \in X_1$ then the cycle
$C'$ contains a path from $y_2$ to either $y_1$ or $y'_1$, whose
length is at least $m-2$, and a path from $x_1$ to either $y_1$ or
$y'_1$, of length at least $k-1$. Thus the total length of $C'$
is at least $(m-2)+(k-1)+4 >m$.
\vspace{0.2cm}

\noindent
{\bf Case 2b:}\, $C'$ contains both $u$ and $v$ and the two
neighbors of $u$ in $C'$ are $y_1 \in Y_1$ and $x_2 \in X_2$.
In this case the path in $C'$ from $v$ to $y_1$ is of length at
least $m-1$ if it does not pass through $X_1$, and at least
$k$ if it passes through $X_1$, and the path from $v$ to $x_2$ is
of length at least $k$ if it does not pass through $X_1$ and of
length at least $m-1$ if it does pass through $X_1$. In all these cases
the length of $C'$ is at least $2+2 \min\{m-1,k\}=2m >m$
(where here we used the assumption that $m<k$).
\vspace{0.2cm}

\noindent
{\bf Case 2c:}\, $C'$ contains both $u$ and $v$ and at least one
edge from $v$ to $X_1$. This is symmetric to either Case 2a or Case
2b.

It thus follows that the number of cycles of length $m$ in $H'$ is
smaller than that number in $H$, contradicting the minimality in
the choice of $H$ and implying that $m \geq \min \{g,k\}$. 
This completes the proof of the theorem.
\hfill $\Box$
\vspace{0.2cm}

\noindent
{\bf Remark:}\, The above proof is constructive, that is, it
provides a polynomial algorithm to find a packing of given graphs
$G_1,G_2$ as above, with the asserted bound on the girth of the
combined graph. Indeed, as long as the girth is too small we can
find a shortest cycle $C$, take in it a vertex $u$ as in the
proof, find a vertex $v$ far from it and swap their roles in the
image of $G_1$. By the argument above this decreases the number of
short cycles by at least $1$. As the total number of such cycles is
less than $n$ by the choice of the parameters 
and by (\ref{e111}), this process terminates in
polynomial time.

\subsection{Directed expanders}

By applying Theorem \ref{t111} repeatedly, starting with a cycle of
length $n$, it follows that for every
$d$ and all large $n$ there is a $2d$-regular graph on $n$ vertices
with girth at least $(1+o(1))\frac{\log n}{ \log (2d-1)}$ which can
be decomposed into $d$ Hamilton cycles. This is a (modest)
strengthening of the result of Erd\H{o}s and Sachs about the
existence of regular graphs of high girth. A more interesting
application of Theorem \ref{t111} is a strengthening of a result
proved in \cite{AMP} about the existence of high-girth directed
expanders. 

\begin{theo}
\label{t131}
For every  prime $p$ congruent to $1$ modulo $4$ and any $n>n_0(d)$ 
there is an explicit
construction of a $2d$-regular graph on $n$ vertices 
with (undirected) girth
at least $(\frac{2}{3}-o(1)) \frac{\log n}{\log (d-1)}$ 
and an orientation of this graph
so that for every two sets of vertices $X,Y$ satisfying
\begin{equation}
\label{e131}
\frac{|X|}{n} \cdot \frac{|Y|}{n} \geq \frac{16}{d}
\end{equation}
there is a directed edge from $X$ to $Y$ and a directed edge from
$Y$ to $X$.
\end{theo}

This improves the estimate on the girth in the result proved in \cite{AMP} 
by a factor of $3$, and also works for all large $n$. The proof
combines Theorem \ref{t111} with an argument from \cite{AMP} and
a recent result proved in \cite{Al4}. An explicit construction
here means that there is a polynomial time deterministic algorithm
for constructing the desired graphs.

\begin{proof}
An $(n,d,\lambda)$-graph is a $d$ regular graph on $n$ vertices in
which the absolute value of any nontrivial eigenvalue is at most
$\lambda$. The graph is Ramanujan if $\lambda=2 \sqrt {d-1}$.
Lubotzky, Phillips and Sarnak \cite{LPS}, and
independently Margulis \cite{Ma} gave, 
for every prime $p$ congruent to $1$ modulo $4$, an explicit
construction of infinite 
families
of $d=p+1$-regular Ramanujan graphs. The girth of these graphs
is at least $(1+o(1))\frac{2}{3}\log_{d-1} n'$, where $n'$ is the
number of vertices. In \cite{Al4} it is shown 
how one can modify these graphs by deleting a set of
appropriately chosen $n'-n$ vertices and by adding edges
among their neighbors to get an $(n,d,2\sqrt{d-1}+o(1))$-graph with
exactly $n$ vertices keeping the girth essentially the same.
Fix such a graph $H$. By Theorem \ref{t111} we can pack two
copies of it $H_1,H_2$ keeping the girth of the combined graph at least
$$
\min \{(1+o(1))\frac{2}{3}\log_{d-1} n, (1+o(1)) \log_{2d-1} n\}
=(1+o(1))\frac{2}{3}\log_{d-1} n,
$$
where here we used the fact that for all admissible $d$,
$$
\frac{2}{3 \log (d-1)} \leq  \frac{1}{\log (2d-1)}.
$$
Let $G$ be the combined graph. Number its vertices $1,2, \ldots ,n$
and orient every edge $ij$ with $i<j$ from $i$ to $j$ if it belongs
to the copy of $H_1$ and from $j$ to $i$ if it belongs to the 
copy of $H_2$.

It is well known (c.f. \cite{AS},  Corollary 9.2.5) that if
$A,B$ are two subsets of an $(n,d,\lambda)$-graph and
$$
\frac{|A||B|}{n^2} > \frac{\lambda^2}{d^2}
$$ 
then there is an edge connecting $A$ and $B$. Let
$X$ and $Y$ be two sets of vertices satisfying (\ref{e131}). 
Let $x$ be the median of $X$ (according to the numbering 
of the vertices), $y$ the median of $Y$. Without loss of generality
assume that $x \leq y$. Let $A$ be the set of all vertices of $X$
which are smaller or equal to $x$, $B$ the set of all vertices of 
$Y$ that are larger or equal to $y$. Then  $|A| \geq |X|/2$ and
$|B| \geq |Y|/2$. Therefore
$$
\frac{|A||B|}{n^2} \geq \frac{|X||Y|}{4 n^2 } \geq 
\frac{4}{d} > \frac{(2\sqrt{d-1}+o(1))^2}{d^2}.
$$ 
Therefore there is an edge of $H_1$ connecting $A$ and $B$
which, by construction, is oriented from $A$ to $B$.
Similarly there is an 
edge of $H_2$ oriented from $B$ to $A$.
This completes the proof.
\end{proof}

\section{Nearly fair representation } 
The approach described here was initiated in discussions with
Eli Berger and Paul Seymour \cite{BS}. 
Let $G=(V,E)$ be a graph and let $P$ be an arbitrary partition
of its set of edges 
into $m$ pairwise disjoint subsets $E_1, E_2, \ldots ,E_m$. 
The sets $E_i$ will be called the color classes of the partition.
For any
subgraph $H'=(V',E')$ of $G$, let
$x(H',P)$ denote the vector $(x_1,x_2, \ldots ,x_m)$, where
$x_i=|E_i \cap E'|$ is the number of edges of $H'$ that lie in
$E_i$. Thus, in particular, $x(G,P)=(|E_1|, \ldots ,|E_m|).$
In a completely fair representation of the sets $E_i$ in $H'$,
each entry $x_i$ of the vector $x(H',P)$ should be equal
to $|E_i| \cdot \frac{|E'|}{|E|}$. Of course such equality can hold
only if all these numbers are integers. But
even when this is not the case the equality may hold up to a small
additive error. 

In this section we are interested in results (and conjectures)
asserting that when $G$ is either the complete graph $K_n$ or the
complete bipartite graph $K_{n,n}$, then for certain graphs $H$
and for any partition $P$ of $E(G)$ into color classes $E_1,\ldots ,E_m$,
there is a subgraph $H'$ of $G$ which is isomorphic to $H$ so that
the vector $x(H',P)$ is close (or equal) to the vector 
$x(G,P)\frac{|E(H')|}{|E(G)|}$. Stein 
\cite{St} conjectured that if  $G=K_{n,n}$ and $P$ is any partition
of the edges of $G$ into $n$ sets, each of size $n$, then 
there is always a perfect matching  $M$ in $G$ satisfying
$x(M,P)=\frac{1}{n} x(G,P)$, that is, a perfect matching containing
exactly one edge from each color class of $P$. This turned out to
be false, a clever counter-example has been given by Pokrovskiy and
Sudakov. In \cite{PS} they describe a partition of the edges of
$K_{n,n}$ into $n$ sets, each of size $n$, so that every perfect
matching misses at least $\Omega( \log n)$ color classes.

In \cite{AABCKLZ}  
it is conjectured that when $G=K_{n,n}$, $P$ is arbitrary,
and $H$ is a matching of
size $n$, then there is always a copy $H'$ of $H$ (that is, a
perfect matching $H'$ in $G$), so that
$$
\|x(H',P)-\frac{1}{n} x(G,P)\|_{\infty} <2.
$$
This is proved in \cite{AABCKLZ} (in a slightly stronger form) 
for partitions $P$
with $2$ or  $3$ color classes. Here we first prove the following, 
showing that when allowing a somewhat  larger additive error
(which grows with the number of colors $m$ but is independent of $n$)
a similar result holds for partitions with any fixed number of
classes.
\begin{theo}
\label{t211}
For any partition $P$ of the edges of the complete bipartite graph
$K_{n,n}$ into $m$ color classes, there is a perfect matching $M$ so that
$$
\| x(M,P)-\frac{1}{n} x(K_{n,n},P) \|_{\infty}  \leq
\| x(M,P)-\frac{1}{n} x(K_{n,n},P) \|_{2} 
< (m-1)2^{(3m-2)/2}.
$$
\end{theo}
It is worth noting that a random perfect matching $M$ typically
satisfies
$$
\| x(M,P)-\frac{1}{n} x(K_{n,n},P) \|_{\infty}  \leq O(\sqrt n).
$$
The main challenge addressed in the theorem is to get an upper
bound independent of $n$.

Theorem \ref{t211} is a special case of a general result which we
describe next, starting with the following definition.
\begin{definition}
\label{d221}
Let $G$ be a graph and let $H$ be a subgraph of it. Call a family of
graphs $\HH$ (which may have repeated members) 
a {\em uniform cover of width $s$ of the pair $(G,H)$}
if every member $H'$ of $\HH$ is a subgraph of $G$ which is isomorphic
to $H$, the number of edges of each such $H'$ which are not edges
of $H$ is at most $s$, every edge of $H$ belongs to the same number
of members of $\HH$, and every edge in $E(G)-E(H)$ 
belongs to the same positive number of members of $\HH$.
\end{definition}
An example of a uniform cover of width $s=2$ for
$G=K_{n,n}$ and $H$ a perfect matching in it is the following.
Let the $n$ edges of $H$ be
$a_ib_i$ where $\{a_1,a_2, \ldots ,a_n\}$ and
$\{b_1, b_2, \ldots ,b_n\}$ are the vertex classes of $G$. Let
$\HH$ be the family of all perfect matchings of $G$ obtained from $H$
by omitting a pair of edges $a_ib_i$ and $a_jb_j$ and by adding 
the edges $a_ib_j$ and $a_jb_i$. The width is $2$, every edge of
$H$ belongs to exactly ${n \choose 2}-(n-1)$ members of $\HH$, and every
edge in $E(G)-E(H)$ belongs to exactly $1$ member of
$\HH$.
\begin{theo}
\label{t213}
Let $G$ be a graph with $g$ edges, 
let $F$ be a
subgraph of it with $f$ edges,
and suppose there is a uniform cover of width $s$ of the
pair $(G,F)$. Then for any partition $P$ of the edges of $G$ into
$m$-subsets, there is a copy $H$ of  $F$ in $G$ so that
$$
\| x(H,P)-\frac{f}{g} x(G,P) \|_{\infty} \leq
\| x(H,P)-\frac{f}{g} x(G,P) \|_{2} \leq (m-1)2^{(m-2)/2}s^m.
$$
\end{theo}

Theorem \ref{t211} is a simple consequence of Theorem \ref{t213}.
A similar simple consequence is the following.
\begin{prop}
\label{p214}
For any partition $P$ of the edges of the complete graph
$K_{n}$ into $m$ color classes, there is a Hamilton cycle $C$ so that
$$
\| x(C,P)-\frac{2}{n-1} x(K_{n},P) \|_{\infty} \leq 
\| x(C,P)-\frac{2}{n-1} x(K_{n},P) \|_{2} 
< (m-1)2^{(3m-2)/2}.
$$
\end{prop}
Similar  statements follow, by the same reasoning,
for a Hamilton  cycle in a complete bipartite 
graph, or for a perfect matching in a complete graph on an even number
of vertices. We proceed to describe a more general application.

For a fixed graph $T$ whose number of vertices $t$ divides $n$, 
a $T$-factor in $K_n$ is the graph consisting of $n/t$ pairwise
vertex disjoint copies of $T$. In particular, when $T=K_2$ this is
a perfect matching. 
\begin{theo}
\label{t215}
For any fixed graph $T$ with $t$ vertices and $q$ edges
and any $m$ there is a
constant $c=c(t,q,m) \leq (m-1)2^{(m-2)/2}(qt)^m$ 
so that for any $n$ divisible by $t$ and for
any partition $P$ of the edges of the complete graph $K_n$ into $m$
subsets, there is a $T$-factor $H$ so that
$$
\| x(H,P)-\frac{2q}{(n-1)t} x(K_{n},P) \|_{\infty} \leq
\| x(H,P)-\frac{2q}{(n-1)t} x(K_{n},P) \|_{2} \leq c.
$$
\end{theo}

\subsection{Proofs}

We start with the proof of Theorem \ref{t213}.

\begin{proof}
Let $P$ be a partition of the edges of $G$ into $m$ color 
classes $E_i$.
Put 
$$
y=(y_1,y_2, \ldots ,y_m)=\frac{f}{g}x(G,P).
$$
Let $H$ be a copy of $F$ in $G$ for which the quantity
$\|y-x\|_2^2 = \sum_{j=1}^m (y_i-x_i)^2$ is minimum where 
$x=(x_1, x_2, \ldots ,x_m)=x(H,P)$. 
Let $\HH$ be a uniform cover of width $s$ of the pair 
$(G,H)$. Suppose each edge of $H$ belongs to $a$ members of $\HH$
and each edge in $E(G)-E(H)$ belongs to $b>0$ such members.
For each member $H'$ of $\HH$, let $v_{H'}$ 
denote the vector of length $m$ defined as follows. For each
$1 \leq i \leq m$, coordinate number
$i$ of $v_{H'}$ is the number of edges in $E(H')-E(H)$ 
colored $i$ minus the number of edges in $E(H)-E(H')$ 
colored $i$. Note that the $\ell_1$-norm of
this vector is at most $2s$  and its sum of coordinates is $0$.
Therefore, its $\ell_2$-norm is at most $\sqrt {2s^2}$.
Note also that $x(H',P)=x(H,P) +v_{H'}$.

We claim that the sum $S$ of all $|\HH|$-vectors $v_{H'}$ for $H' \in
\HH$ is a positive multiple of the vector $(y-x)$. Indeed, each
edge in $E(G)-E(H)$ is covered by $b$ members of $\HH$, and each
edge of $E(H)$ is covered by $a$  members of $\HH$. In the sum
$S$ above this contributes to the coordinate corresponding to color
number $i$, $b$ times the number of edges of color $i$ in 
$E(G)-E(H)$ 
minus $(|\HH|-a)$ times the number of edges of color
$i$ in $H$. Equivalently, this is
$b$ times the number of all edges of $G$ colored $i$ minus
$(|\HH|+b-a)$ times the number of edges of $H$ colored $i$.
Since the sum of coordinates of each of the vectors
$v_{H'}$ is zero, so is the sum of coordinates of
$S$, implying that $bg=(|\HH|+b-a) f$, that is,
$|\HH|+b-a=\frac{g}{f}b$. Since $\frac{g}{f} y=x(G,P)$ 
this implies that 
$S=\frac{bg}{f}(y-x)$, proving the claim.

Since the vector $y-x$ is a linear combination
with positive coefficients of the 
vectors $v_{H'}$  it follows,
by Carath\'eodory's Theorem for cones,
that there exists a set $L$ of linearly independent vectors
$v_{H'}$ so that $y-x$ is
a linear combination with positive coefficients of them.
Indeed, starting
with the original expression of
$y-x$  mentioned above, as long as there is a linear dependence
among the vectors $v_{H'}$ participating in the combination
with nonzero (hence positive) coefficients, we can subtract an
appropriate multiple of this dependence and ensure that at least
one of the nonzero coefficients vanishes  and all others stay
non-negative (positive, after omitting all the ones with coefficient
$0$). As each vector $v_{H'}$ has $m$ coordinates and their sum is
$0$, it follows that $|L| \leq m-1$. 

We can now solve the system of linear equations $y-x=\sum z_{H'} v_{H'}$
with the variables $z_{H'}$ for $v_{H'} \in L$.
Note that it is enough to
consider any $|L|\leq m-1$ coordinates of $y-x$ and solve the system
corresponding to these coordinates. 
By Cramer's rule applied to this system each $z_{H'}$ is a ratio of
two determinants. The denominator is a determinant of a nonsingular matrix 
with integer coefficients, and its absolute value is thus at least $1$.
The numerator is also a determinant, and  
by Hadamard's Inequality its absolute value is at most the product of
the $\ell_2$-norms of the columns of the corresponding matrix. 
The norm of one column is at
most $\|y-x\|_2$ (this can be slightly improved by
selecting the $|L|$-coordinates with the smallest $\ell_2$-norm,
but we do not include this slight improvement here). Each other column
has norm at most $(2s^2)^{1/2}$. Therefore
each coefficient $z_{H'}$ satisfies
$0 \leq z_{H'} \leq \|y-x\|_2 (2s^2)^{(m-2)/2}.$
By taking the inner product with $y-x$ we get
$$
\|y-x\|_2^2 =\sum_{v_{H'} \in L} z_{H'} \langle y-x,v_{H'} \rangle
$$
$$
\leq \sum_{v_{H'} \in L, \langle y-x, v_{H'} \rangle >0}
z_{H'} \langle y-x,v_{H'} \rangle
$$
$$
\leq (m-1) \|y-x\|_2 (2s^2)^{(m-2)/2} \max \langle y-x,v_{H'} \rangle.
$$

Therefore, there  is a $v_{H'}$ so that
$$
\frac{\|y-x\|_2}{(m-1) (2s^2)^{(m-2)/2} }=
\frac{\|y-x\|_2^2}{(m-1) (2s^2)^{(m-2)/2} \|y-x\|_2} \leq
\langle y-x,v_{H'} \rangle,
$$
that is,
\begin{equation}
\label{e221}
\|y-x\|_2 \leq (m-1) (2s^2)^{(m-2)/2} \langle y-x,v_{H'} \rangle
=(m-1)2^{(m-2)/2} s^{m-2}  \langle y-x,v_{H'} \rangle.
\end{equation}
By the minimality of $\|y-x\|_2^2$ 
$$
\|x+v_{H'}-y\|_2^2 =\|x-y\|_2^2 -2 \langle y-x,v_{H'} \rangle
+\|v_{H'}\|_2^2 \geq \|x-y\|_2^2,
$$
implying that
$$
2 s^2 \geq \|v_{H'}\|_2^2 \geq 2 \langle y-x,v_{H'} \rangle.
$$
Plugging in (\ref{e221}) we get
$$
\|y-x\|_2 \leq (m-1)2^{(m-2)/2} s^{m},
$$
and the desired results follows since
$ \|y-x\|_{\infty} \leq \|y-x\|_2$.
\end{proof}

The assertions of Theorem \ref{t211} and Proposition \ref{p214}
follow easily from Theorem \ref{t213}. Indeed, as described above
there is a simple uniform cover of width $s=2$ for the pair
$(K_{n,n},M)$ where
$M$ is a perfect matching. There is also a similar uniform cover
$\HH$
of width $s=2$ for the pair $(K_n,C)$ where $C$ is a Hamilton
cycle. The $n(n-3)/2$ members of $\HH$ are all Hamilton cycles
obtained from $C$ by omitting two nonadjacent edges of it and by
adding the two edges that connect the resulting pair of paths to
a cycle.

To prove Theorem \ref{t215}
we need the following simple lemma.
\begin{lemma}
\label{l221}
Let $T$ be a fixed graph with $t$ vertices and $q$ edges, suppose
$t$ divides $n$ and let $H$ be a $T$-factor in $K_n$. Then there
is a uniform cover of width at most $qt$ of the pair $(K_n,H)$.
\end{lemma}
\begin{proof}
Let $H$ be a fixed $T$-factor in $K_n$, it consists of
$p=n/t$ (not necessarily connected) vertex disjoint copies of $T$
which we denote by $T_1,T_2, \ldots ,T_p$. Let $\HH_1$ be the set of
all copies  $H'$ of the $T$-factor obtained from $H$ by replacing one
the copies $T_i$
by another copy of $T$ on the same set of vertices, in all 
possible $t! $ ways. Note that if $T$ has a nontrivial automorphism
group some members of $\HH_1$ are identical, and $\HH_1$
is a multiset.
By symmetry it is clear that each edge
of $H$ belongs to the same number of members of $\HH_1$. Similarly,
each edge connecting two vertices of the same $T_i$ which does not
belong to $H$ lies in the same positive number of members of $\HH_1$. Beside
these two types of edges, no other edge of $K_n$ is covered by any
member of $\HH_1$. Let $\HH_2$ be the (multi)-set of all copies of
the $T$-factor obtained from $H$ by choosing, in all possible ways,
$t$  of the copies of $T$, say, $T_{i_1}, T_{i_2}, \ldots
,T_{i_t}$, removing them, and replacing them by all possible
placements of $t$ vertex disjoint copies of $T$ 
where each of the newly placed copies
contains exactly one vertex of each $T_{i_j}$. Again by symmetry it
is clear that each edge of $H$ belongs to the same number of
members of $\HH_2$. In addition, each
edge of $K_n$  connecting vertices from
distinct copies of $T$ in $H$
belongs to the same (positive) number of members of $\HH_2$. No
other edges of $K_n$ are covered by any $H' \in \HH_2$. 
It is now simple
to see that there are two integers $a$, $b$, so that the multiset
$\HH$ consisting of $a$ copies of each member of $\HH_1$ and $b$
copies of each member of $\HH_2$ is a uniform cover of the pair
$(K_n,H)$. The width of this cover is clearly $qt$, as every member
of $\HH_2$ contains $qt$ edges not in $E(H)$, and every member of
$\HH_1$ contains at most $2q$ edges not in $E(H)$. This completes
the proof.
\end{proof}
The assertion of Theorem \ref{t215} clearly follows from
the last Lemma together with Theorem \ref{t213}. 

\subsection{Concluding remarks and open problems}
\begin{itemize}
\item
The statement of Theorem \ref{t215} holds for any graph $H$
consisting of $n/t$ (not necessarily connected) 
vertex disjoint components, each having $t$ vertices
and $q$ edges. The proof applies with no need  to assume that all these
components are isomorphic.
\item
The proof of Theorem \ref{t213} is algorithmic in the sense that if
the cover $\HH$  is given then one can find, in time polynomial
in $n$ and $|\HH|$,
a copy $H$ of $F$ satisfying the conclusion. Indeed, the proof
implies that  as long as we have a copy $H$ for which
the conclusion does not hold, there is a member $H' \in \HH$ for
which $\| x(H',P) - \frac{f}{g} x(G,P) \|_2^2$  is strictly
smaller than $\| x(H,P) - \frac{f}{g} x(G,P) \|_2^2$. 
By checking all members of $\HH$ we can find an $H'$ for which 
this holds.
As both these
quantities are non-negative rational numbers smaller than
$n^4$ with denominator 
$g^2<n^4$, this process 
terminates in a polynomial number of steps. We make no attempt to 
optimize the number of steps here.
\item
The results can be extended to $r$-uniform hypergraphs by a
straightforward 
modification of the proofs.
\item
There are graphs $H$ for which no result 
like those proved above
holds when $G$ is either a complete or a complete bipartite
graph even if the number of colors is small. A simple example
is when $G=K_{2n}$, $H=K_{1,2n-1}$ and $m=3$. The edges of $K_{2n}$ can be
partitioned into two vertex disjoint 
copies of $K_n$ and a complete bipartite graph
$K_{n,n}$. For this partition, every copy of the star  $H$ misses
completely one of the color classes, although it's fair share in it
is roughly a quarter of its edges. More generally, let  $H$ be any
graph with a vertex cover of size smaller than $m-1$ (that is, $H$
contains a set of less than $m-1$ vertices touching all its edges).
Consider a partition of the edges of the complete graph $K_n$ into
$m-1$ pairwise vertex disjoint copies of the complete graph on
$\lfloor n/(m-1) \rfloor$ vertices, and an additional class
containing all the remaining  edges. Then any copy of $H$ in this
graph cannot contain edges of all those $m-1$ complete subgraphs,
as the edges of the copy can be covered by less than $m-1$ stars.
It is easy to see that a 
similar example exists for $G=K_{n,n}$ as well.
\item
The discussion here suggests the following conjecture.
\begin{conj}
\label{c231}
For every $d$ there exists a $c(d)$ so that for any graph $H$ with
at most $n$ vertices and maximum degree at most $d$ and for any
partition $P$ of the edges of $K_n$ into $m$ color classes, there is
a copy $H'$ of $H$ in $K_n$ so that 
$$
\| x(H',P)-\frac{|E(H)|}{E(K_n)|} x(K_{n},P) \|_{\infty} \leq
c(d).
$$
\end{conj}
The analogous conjecture for bipartite bounded-degree graphs $H$
with at most $n$ vertices in each color class and for partitions of
the edges of $K_{n,n}$ is also plausible. Note that the conjecture
asserts that the same error term $c(d)$ should hold for any number of
colors $m$. Note also that $c(d)$ must be at least $\Omega(d)$ as shown by
the example of a star $H=K_{1,d}$ and the edge-coloring of $K_{2n}$ with 
$m=3$ colors described above.
\end{itemize}

\section{The choice number of complete multipartite graphs with
equal color classes}
The choice number  of a graph $G$ is the 
smallest integer $s$ so that for any
assignment of a list of $s$ colors to each vertex of $G$
there is a proper coloring of $G$ assigning to each vertex a color
from its list. This notion was introduced in \cite{Vi}, \cite{ERT}. 
Let $K_{m*k}$ denote the complete $k$-partite graph with 
$k$ color classes, each of size $m$. Several researchers
investigated the choice number
$ch(K_{m*k})$ of this graph.  Trivially $ch(K_{1*k})=1$ as
$K_{1*k}$ is a $k$-clique. In \cite{ERT} it is proved that
$ch(K_{2*k})=k$. Kierstead  \cite{Ki} proved that
$ch(K_{3*k})=\lceil (4k-1)/3 \rceil$ and
in \cite{KSW} it is proved that
$ch(K_{4*k})= \lceil (3k-1)/2 \rceil.$

In \cite{ERT} it is shown that as $m$ tends to infinity
$ch(K_{m*2})=(1+o(1))\log_2 m$. In \cite{Al} the author shows
that there are absolute constants $c_1,c_2>0$ so that
$c_1 k \ln m \leq ch(K_{m*k}) \leq c_2 k \ln m$ for all $m$ and
$k$. In \cite{GK} it is proved that for fixed $k$,
as $m$ tends to infinity, $ch(K_{m*k})=(1+o(1)) \frac{\ln m}{\ln
(k/(k-1)}$ and in  \cite{Sh} it is proved that if both
$m$ and $k$ tend to infinity and $\ln k =o(\ln m)$ then
$ch(K_{m*k})=(1+o(1)) k \ln m$. Our first result here is that 
the assumption
that $\ln k =o(\ln m)$ can be omitted, obtaining the asymptotics of
$ch(K_{m*k})$ when $m$ and $k$ tend to infinity (with no assumption
on the relation between them).
\begin{theo}
\label{t311}
If $m$ and $k$ tend to infinity then
$$
ch(K_{m*k})=(1+o(1))k \ln m.
$$
\end{theo}
The proof is probabilistic, similar to the one in \cite{Al}, 
where the main additional argument is in the proof of the upper
bound for values of $k$ which are much bigger than $m$. 

Our second result is the following.
\begin{theo}
\label{t312}
For any fixed integer $m \geq 1$ the limit
$$
\lim_{k \rightarrow \infty} \frac{ch(K_{m*k})}{k}
$$
exists (and is $\Theta(\ln m)$).
\end{theo}

For $m \geq 1$, let $c(m)$ denote the above limit.
By the known results stated above
$c(1)=c(2)=1$, $c(3)=4/3$, $c(4)=3/2$ and
$c(m)=(1+o(1)) \ln m$. The problem of determining $c(m)$ precisely
for every $m$ seems very difficult.

We prove Theorem \ref{t311} without trying to optimize the error terms.
To simplify the presentation, we omit all floor and ceiling signs
whenever these are not crucial.

\subsection{The upper bound}

\begin{prop}
\label{p321}
For every $m,k \geq 2$
$$
ch(K_{m*k}) \leq k (\ln m + \ln \ln m+20).
$$
\end{prop}

\noindent
{\bf Proof:}\, 
Since $\ln m + \ln \ln m +20  \geq 20$  for all $m \geq 2$ 
we may and will assume that $m > 20$.  We consider two possible cases.
\vspace{0.2cm}

\noindent
{\bf Case 1:}\, $k \leq 10 \ln m$.

In this case we show that lists of size 
$s=k(\ln m + \ln \ln m +3)$ suffice. Let $G=K_{m*k}=(V,E)$, and
suppose we  assign a list $S_v$ of  colors to each vertex $v \in
V$, where $|S_v|=s$ for all $v$. Let $S=\cup_{v \in V} S_v$
be the union of all lists. Let $S=T_1 \cup T_2 \ldots \cup T_k$ be
a random partition of all colors in $S$ into $k$ pairwise disjoint
subsets, where each color $x \in S$ is assigned, randomly, uniformly
and independently, to one of the subsets $T_j$. We obtain a proper
coloring of $G$ by coloring each vertex $v$ that lies in color
class number $j$ by a color from $S_v \cap T_j$. Clearly, if there is
indeed such a color for each vertex, then the resulting coloring
is  proper. The probability that for a fixed vertex $v$ the
above fails is exactly
$$
(1-\frac{1}{k})^{|S_v|} \leq e^{-(\ln m + \ln \ln m +3)}
<\frac{1}{e^3 m \ln m } < \frac{1}{km}.
$$
As there are $mk$ vertices, 
the probability that there is a vertex  for which the
above fails is smaller than $1$, completing the proof in this
case.
\vspace{0.2cm}

\noindent
{\bf Case 2:}\ $ k > 10 \ln m$. 

Note that since by assumption 
$m>20$ this implies that $k \geq 30$.
In this case we show that lists of size $s=k(\ln m+20)$ suffice.
Let  $G=K_{m*k}=(V,E)$, and
suppose we assign a list $S_v$ of  $s$ colors to each vertex $v \in
V$. As before, let $S=\cup_{v \in V} S_v$
be the union of all lists. Our strategy now is to first define a
set of reserve colors $R$, these colors will be used to
assign colors to the vertices that will not  be colored 
by the procedure applied in Case 1. Let $R$ be a random subset of
$S$ obtained by picking each color in $S$ to lie in $R$
with probability $p=\frac{10}{\ln m + 20}$, where all choices are
independent. For a fixed vertex $v$, the random variable
$|S_v \cap R|$ is a Binomial random variable with expectation
$sp=10k$. By the standard estimates for Binomial distributions
(see, e.g., \cite{AS}, appendix A, Theorems A.1.11 and A.1.13),
the probability that this random variable is smaller than $k$ is
less than $e^{-10k/8}$ and the probability it is larger than 
$20k$ is less than $e^{-10k/14}$. Thus the probability it is not
between $k$ and $20k$  is less than
$$
2e^{-10k/14}<2e^{-k/3} e^{-k/3} < 2\frac{1}{2k}\frac{1}{m^3}
< \frac{1}{mk},
$$
where here we used the fact that $k \geq 10 \ln m$ to conclude
that $e^{-k/3} < \frac{1}{m^3}$ and the fact that $k \geq 30$ to
conclude that $e^{-k/3} < \frac{1}{2k}.$
It follows that with positive probability
$k \leq |S_v \cap R| \leq 20 k$ for every vertex $v \in V$. 
Fix a set of colors $R$ for which this holds. Now proceed as 
in Case 1.
Let $S-R=T_1 \cup T_2 \ldots \cup T_k$ be
a random partition of all colors in $S-R$ into $k$ pairwise disjoint
subsets, where each color in $S-R$ is assigned, randomly, uniformly
and independently to one of the subsets $T_j$. If a vertex 
$v$ of $G$ lies in color class
number $j$, and $S_v \cap T_j \neq \emptyset$, then color it by an
arbitrary color in this intersection $S_v \cap T_j$.
The probability that $v$ fails to have such a color is
$$
(1-\frac{1}{k})^{|S_v-R|} \leq (1-\frac{1}{k})^{k \ln m} \leq
\frac{1}{m},
$$
where here we used the fact that $|S_v \cap R| \leq 20k$ for all
$v$.
By linearity of expectation, the  expected number of uncolored
vertices at this stage is at most $k$, hence we can fix a
splitting $T_1,\cdots, T_k$ as above  so that there are at most
$k$ uncolored vertices. But now we can color these vertices one by
one using the reserve colors. Since for each such vertex $u$,
$|S_u \cap R| \geq k$, each of these vertices has at least
$k$ colors of $R$ in its list and thus we will be able to assign to
it a color that differs from all colors of $R$ assigned to
previous vertices. This completes the proof of the upper bound.
\hfill $\Box$

\subsection{The lower bound}
The proof of the lower bound is  essentially the one in 
\cite{Al}, with a more careful computation and choice of parameters. 
For completeness, we sketch the details. 
\begin{prop}
\label{p322}
There exists an $m_0$ so that for all $m > m_0$ and every $k$
$ch(K_{k*m}) > t$ where
$$
t=(k-1-\frac{k}{\ln m}) (\ln m - 4 \ln \ln m) 
(1-\frac{\ln m}{m})~ (~=(1+o(1)) k \ln m),
$$
where the $o(1)$-term tends to zero as $m$ and $k$ tend to
infinity.
\end{prop}

\noindent
{\bf Proof:}\,
We consider two possible cases.
\vspace{0.2cm}

\noindent
{\bf Case 1:}\,  $k \leq m $.

In this case  we prove that $ch(K_{k*m})>s$, where
$$
s= (k-1-\frac{k}{\ln m}) (\ln m - 4 \ln \ln m) ~ (~=(1+o(1)) k \ln m).
$$
Let $S$ be a set of $k (\ln m)^2$ colors, and let
$S_1,S_2, \ldots ,S_m$ be $m$ random subsets of $S$, each chosen
independently and uniformly among all subsets of cardinality
$s$ of $S$, where $s$ is as above. We claim  that with positive
probability there is no subset of $S$ of cardinality 
at most  $|S|/k = (\ln m)^2$ that intersects all 
subsets $S_i$. This claim suffices to prove the assertion of the
proposition in this case. Indeed, we simply assign the $m$ vertices
in each color class of $G$ the $m$ lists $S_i$. If there would have
been a proper coloring of $G$ assigning to each vertex a color 
from its list, then the set of all colors assigned to  vertices 
in one of the color classes of $G$ must be of size at most $|S|/k$ and
it must intersect all
lists $S_i$, contradiction. It thus suffices to prove the claim.

Fix a set $T$ of $(\ln m)^2$ colors. The probability that a random
subset of size $s$ of $S$ does not intersect $T$ is
$$
\frac{{{|S|-|T|} \choose s}}{{{|S|} \choose s}}
=\frac{{{(k-1) \ln^2 m} \choose s}}{{{k \ln^2 m } \choose s}}.
$$
This quantity is at least
$$
(\frac{(k-1) \ln^2 m - k \ln m}{k \ln^2 m- k \ln m}) ^s
=(1-\frac{1}{k(1-1/\ln m)})^{[k(1-\frac{1}{\ln m})-1](\ln m -4 \ln
\ln m)}
$$
$$
\geq (\frac{1}{e})^{\ln m-4 \ln \ln m} =\frac{\ln^4 m}{m},
$$
where here we used the fact that for every $q>1$, $(1-1/q)^{q-1}
\geq \frac{1}{e}$.
Therefore, the probability that none of the $m$
random sets $S_i$ misses $T$ is at most
$$
(1-\frac{\ln^{4} m}{m})^m < e^{-\ln^{4}m}.
$$
As the number of choices for $T$ is only
$$
{{k \ln^2 m} \choose {\ln^2 m}} \leq (ek)^{\ln^2 m}
 \leq e^{(1+o(1))\ln^3 m},
$$
where here we used the assumption that $k \leq m$, the desired claim
follows, completing the proof of Case 1.
\vspace{0.2cm}

\noindent
{\bf Case 2:}\, $k \geq m$.

In this case, take first the previous construction with $m$ and
$k'=\ln m$. Replace $k$ by the largest integer $k''$ which is
at most $k$ and is divisible by $k'$, that is: 
$k''=k'\lfloor k/k' \rfloor$. Note that as $k \geq m$ and $k' = \ln
m$, $k'' \geq k (1-\frac{\ln m}{m})$. Now replace in the
construction for $k'=\ln m$ every color by a group of 
$k''/k'$ colors, where all groups are pairwise disjoint,
to get $m$ lists, each of size
$(1+o(1))k'' \ln m=(1+o(1)) k \ln m$, in  a set of size
$k'' \ln^2 m$,
so that no 
subset of size $\ln^2 m$, that is,
a fraction of $1/k''$ of the colors, intersects
all of them. This shows, as before. that
$ch(K_{m*k''}) > (1+o(1)) k'' \ln m=(1+o(1)) k \ln m$, 
and as
$ch(K_{m*k})$ can be only larger (since it contains $K_{m*k''})$ as a
subgraph), this completes the proof.
\hfill $\Box$

\subsection{The existence of the limit}

In this subsection we prove Theorem \ref{t312}. A natural way to try
and prove it is to show that for every
fixed $m$, the function $f(k)=ch(K_{m*k})$ is either sub-additive
or super-additive. In theses cases the existence of the limit would
follow from Fekete's Lemma. Unfortunately this function is not
always super-additive, as shown by the case $m=3$, since 
$ch(K_{3*2})=3$ and 
$$
ch(K_{3*2k})= \lceil (8k-1)/3 \rceil <3k.
$$

Similarly, the function is not always sub-additive, as shown by the
case of large $m$, where 
$ch(K_{m*2})=(1+o(1)) \log_2 m$
and for large $k$,
$$
ch(K_{m*2k})= (1+o(1)) 2k \ln m > (1+o(1)) k \log_2 m.
$$
Still we show that the limit exists by proving that the
above function is nearly sub-additive.

We need the following technical lemma.
\begin{lemma}
\label{l323}
There is a positive integer $s_0$, so that for every integer
$s>s_0$ the following holds. For every real
$c$ satisfying $1/3 \leq c \leq 2/3$ and for every integer $t \geq
2$:
$$
c[(s^{1/3}+3)t^{1/3}-3]^3 -c [(s^{1/3}+3)t^{1/3}-3]^2
\geq
[(s^{1/3}+3)(ct)^{1/3}-3]^3.
$$
\end{lemma}

\noindent
{\bf Proof:}\, 
Put 
$$
X=c^{1/3}[(s^{1/3}+3)t^{1/3}-3],~~~Y=(s^{1/3}+3)(ct)^{1/3}-3.
$$
Then the above inequality is equivalent to the statement
$$
X^3 -c^{1/3} X^2 \geq Y^3,
$$
that is, to
$$
(X-Y)(X^2+XY+Y^2) \geq c^{1/3} X^2.
$$
Since $1/3 \leq c \leq 2/3$, we have $0.69<c^{1/3}<0.88$.
Thus $X-Y=3-3c^{1/3}>0.36$. For sufficiently large $s$,
$X>Y>0.9X>0$ and thus $XY >0.9 X^2$ and $Y^2>0.8 X^2$.
Therefore
$$
(X-Y)(X^2+XY+Y^2) > 0.36 \cdot 2.7 X^2=0.972 X^2
> 0.88 X^2 > c^{1/3} X^2.
$$
This completes the proof. \hfill $\Box$

We also need the following simple corollary of Chernoff's
Inequality (see, e.g., \cite{AS}, Appendix A.)
\begin{lemma}
\label{l324}
There exists an $s_0>0$ so that for every $s >s_0$, every integer
$t \geq 2$ and every real $c$ satisfying $1/3 \leq c \leq 2/3$,
the probability that the Binomial random variable with parameters
$[(s^{1/3}+3)t^{1/3}-3]^{3}$ and $c$ is at most
$$
c[(s^{1/3}+3)t^{1/3}-3]^3 -c [(s^{1/3}+3)t^{1/3}-3]^2
$$
is smaller than $\frac{1}{(st)^2}$.
\end{lemma}

\noindent
{\bf Proof:},
By Chernoff this probability is smaller than
$$
e^{-\Omega((st)^{1/3})}.
$$
\hfill  $\Box$

Using the above, we prove the following.
\begin{prop}
\label{p325}
For every fixed $m$ there exists $k_0=k_0(m)$ so that for all 
$k>k_0$ the following holds. If $ch(K_{m*k})=s$ then
for every integer $t \geq 1$
$$
ch(K_{m*kt}) \leq [(s^{1/3}+3)t^{1/3}-3]^3.
$$
\end{prop}
{\bf Proof:}\,
Since trivially $ch(K_{m*k}) \geq ch(K_{1*k})=k$, we can choose
$k_0>m$ so that for $k>k_0$,
$s=ch(K_{m*k})$ is sufficiently large to ensure that
the assertions of Lemma \ref{l323} and Lemma \ref{l324} hold. 
Note also that for this $k_0$, $s >m$.
With this $k_0$ we prove the above by induction on $t$.
For $t=1$ there is nothing to prove. Assuming the result
holds for all integers $t'<t$ we prove it for $t$.
Let $G=K_{m*kt}=(V,E)$ have the $kt$ color classes
$U_1,U_2, \cdots U_{kt}$, and suppose we have a list  $L_v$
of $[(s^{1/3}+3)t^{1/3}-3]^3$ colors assigned to each vertex
$v \in V$. Put $t_1=\lfloor t/2 \rfloor$,
$t_2=\lceil t/2 \rceil$ and split $V$ into two disjoint 
sets $V_1,V_2$, where
$V_1$ consists of all vertices in the first $t_1k$ color classes
$U_j$ and $V_2$ consist of all vertices in the last $t_2$ color
classes $U_j$. Let $G_1$ be the induced subgraph of $G$ on $V_1$
and
$G_2$ the induced subgraph of $G$ on $V_2$. Thus $G_1$ is a copy of
$K_{m*kt_1}$ and $G_2$ is a copy of $K_{m*t_2}$.

Let $S \cup_{v \in V} L_v$ be the set of all colors,
and let $S=S_1 \cup S_2$ be a random partition of it into two
disjoint sets, where each color in $S$ is chosen, randomly and
independently, to lie in $S_1$ with probability $t_1/t$ and to lie
in $S_2$ with probability $t_2/t$.

Our objective is to use only the colors of $S_1$ for the vertices
in $G_1$ and only those of $S_2$ for the vertices in $G_2$.
Note that $1/3 \leq t_1/t \leq t_2/t \leq 2/3$. For each vertex
$v \in V_1$ the set $L_v \cap S_1$ of colors in $S_1$ that belong
to the list of $v$ is of size which is a binomial random variable
with parameters 
$[(s^{1/3}+3)t^{1/3}-3]^3$ and $t_1/t$. Therefore, by Lemma
\ref{l324} the probability
that this size is smaller than 
$(t_1/t)[(s^{1/3}+3)t^{1/3}-3]^3 -(t_1/t) [(s^{1/3}+3)t^{1/3}-3]^2$
is less than $\frac{1}{(st)^2}$. By the same reasoning the
probability that for a vertex $u \in V_2$ the size of
$L_u \cap V_2$ is smaller than
$(t_2/t)[(s^{1/3}+3)t^{1/3}-3]^3 -(t_2/t) [(s^{1/3}+3)t^{1/3}-3]^2$
is less than $\frac{1}{(st)^2}$. As $s>m, s \geq k$ 
the total number of vertices is 
smaller than $kst<(st)^2$ and hence with positive probability this
does not
happen for any vertex. By Lemma \ref{l323} in this case each vertex
of $G_1$ still has at least 
$[(s^{1/3}+3)(t_1)^{1/3}-3]^3$ colors in its list (restricted to
the colors in $S_1$), and a similar statement holds for the vertices
of $G_2$. We can now fix a partition $S=S_1 \cup S_2$ for which this holds 
and apply induction to color $G_1$ by the colors from $S_1$ and
$G_2$ by the colors from $S_2$, completing the proof. 
\hfill $\Box$
\vspace{0.2cm}

\noindent
{\bf Proof of Theorem \ref{t312}:}\,
Fix an integer $m \geq 1$.
By the result of \cite{Al} stated in Section 1, 
$$
\lim \inf_{k \rightarrow \infty} \frac{ch(K_{m*k})}{k}=q
$$ 
exists (and is $\Theta(\ln m)$). Fix a small $\epsilon>0$
and let $k>k_0$ be a large integer, where $k_0$ is as in
Proposition \ref{p325}, so that 
$$
\frac{ch(K_{m*k})}{k} \leq q +\epsilon.
$$
Put $s=ch(K_{m*k})$. Then $s \leq k(q+\epsilon)$.
By Proposition \ref{p325} for every integer
$t \geq 1$,
$$
ch(K_{m*kt}) \leq [(s^{1/3}+3)t^{1/3}-3]^3 
< [s^{1/3}e^{3/s^{1/3}}t^{1/3}]^{3}=ste^{9/s^{1/3}}.
$$
Suppose, further, that $k$ is chosen to be sufficiently large to
ensure that 
$$
e^{9/k^{1/3}}<(1+\epsilon).
$$ 
As $s=ch(K_{m*k}) \geq k$
in this case we have also
$$
e^{9/s^{1/3}}<(1+\epsilon).
$$
Therefore, for every integer $t \geq 1$
$$
ch(K_{m*kt}) \leq 
ste^{9/s^{1/3}} <
k(q+\epsilon)t(1+\epsilon).
$$
It follows that for every large integer $p$,
$$
ch(K_{m*p}) \leq
k(q+\epsilon) \lceil p/k \rceil (1+\epsilon)
\leq 
k(q+\epsilon) (p+k)/k (1+\epsilon).
$$
Thus
$$
\frac{ch(K_{m*p})}{p} 
\leq k(q+\epsilon) (p+k)/(pk) (1+\epsilon)
$$
which for sufficiently large $p$ is at most, say,
$$
(q+\epsilon)(1+\epsilon)^2.
$$
Since, by the result in \cite{Al}, $q =\Theta( \ln m)$ 
and $\epsilon>0$ can be chosen to be arbitrarily small 
this implies that 
$$
\lim \sup_{p \rightarrow \infty} \frac{ch(K_{m*p})}{p}
\leq q =
\lim \inf_{p \rightarrow \infty} \frac{ch(K_{m*p})}{p},
$$
completing the proof. 
\hfill $\Box$

\section{On vector balancing}
Let $p$ be a prime, let $w_1=e^{2 \pi i/p}$ be the $p$th primitive 
root of unity, and define $w_j=w_1^j$ for $0 \leq j \leq p-1$. Let
$n$  be an integer divisible by $p$, and let $B$ be the set of all
$p^n$ vectors of length $n$ in which each coordinate is in the set
$\{1,w_1, \ldots ,w_{p-1} \}$. Let $K(n,p)$ denote the minimum $k$
so that there exists a set $\{v_1, v_2, \ldots ,v_k\}$ of members of
$B$ such that for every $u \in B$ there is some $1 \leq j \leq k$
so that the scalar inner product $v_i \cdot u=0$.

Heged\H{u}s \cite{Heg} proved that for every prime $p$ and $n$
divisible by $p$, $K(n,p) \geq (p-1)n$, extending a result of
\cite{ABCO} where the statement is proved for $p=2$. He also
conjectured that
equality always holds, as is the case for $p=2$, by a simple construction
of Knuth (c.f. \cite{ABCO}).
Our first observation here is that this conjecture
is (very) false for every prime $p \geq 5$ and large $n$.
\begin{prop}
\label{p411}
For every prime $p$ and every $n$ divisible by $p$
\begin{equation}
\label{e411}
K(n,p) \geq \frac{p^n [(n/p)!]^p}{n!}.
\end{equation}
Therefore, for every fixed $p$ and large $n$
\begin{equation}
\label{e412}
K(n,p) \geq  (1+o(1)) \frac{(2 \pi)^{(p-1)/2}}{p^{p/2}} \cdot
n^{(p-1)/2}.
\end{equation}
\end{prop}

The proof of Heged\H{u}s is based on Gr\"obner basis methods. In
particular, he established the following result.
\begin{theo}[\cite{Heg}]
\label{t412}
Let $p$ be a prime and let $P(x)=P(x_1,x_2, \ldots ,x_{4p})$ be a
polynomial over $Z_p$ which vanishes over all $\{0,1\}$ vectors of
Hamming weight $2p$ and suppose that there is a $\{0,1\}$-vector 
$z$ of Hamming weight $3p$ so that $P(z) \neq 0$ (in $Z_p$). 
Then the degree of $P$ is at least $p$.
\end{theo}
An elementary proof of this lemma, due to S. Srinivasan, is
given in \cite{AKV}. Here we describe a variant of this proof
providing a very short derivation
of this lemma from the Combinatorial Nullstellensatz proved in
\cite{Al00}, which is the following.
\begin{theo}
\label{t413}
Let $F$ be an arbitrary field, and let $f=f(x_1, \ldots ,x_n)$
be a polynomial in $F[x_1, \ldots ,x_n]$. Suppose the degree
$deg(f)$ of $f$ is $\sum_{i=1}^n t_i$, where each $t_i$ is a
nonnegative integer, and suppose the coefficient of
$\prod_{i=1}^n x_i^{t_i}$ in $f$ is nonzero. If
$S_1, \ldots ,S_n$ are subsets of $F$ with $|S_i|>t_i$,
then there are $s_1 \in S_1, s_2 \in S_2, \ldots, s_n \in S_n$
so that
$$
f(s_1, \ldots ,s_n) \neq 0.
$$
\end{theo}
\subsection{Proofs}

\noindent
{\bf Proof of Proposition \ref{p411}:}\, 
Let $M$ be the collection of all vectors in $B$ in which each
$w_i$ appears in
exactly $n/p$ coordinates and let
$$
m=|M|=\frac{n!}{[(n/p)!]^p}
$$ 
be its cardinality. We claim that $M$ is the set of all vectors in
$B$ that are orthogonal to the vector ${\bf j}=(1,1, \ldots ,1) \in B$.
Indeed, it is a well known consequence of Eisenstein's
criterion that the minimal polynomial of $w_1$ over the rationals
is the polynomial $1+x+x^2 + \cdots +x^{p-1}$. Therefore, if 
$\sum_{i=0}^{p-1} \alpha_i w_i=0$ for some integers $\alpha_i$,
then the polynomial $1+x+x^2 + \cdots +x^{p-1}$ divides
$\sum_{i=0}^{p-1} \alpha_i x^i$, implying that all the coefficients 
$\alpha_i$ are equal. This implies the assertion of the claim.

By the claim, the number of vectors in $B$  orthogonal to
${\bf j}$ is exactly $m$, and this is clearly also the
number of vectors in $B$ orthogonal to any other fixed member 
of $B$. It follows that if each vector in $B$ is orthogonal to at
least one vector in a subset of cardinality $k=K(n,p)$  of $B$,
then $k \geq p^n/m$, implying (\ref{e411}). The estimate in
(\ref{e412}) follows from (\ref{e411}) by Stirling's Formula.
$\Box$
\vspace{0.2cm}

\noindent
{\bf Proof of Theorem \ref{t412}:}\, 
Without loss of generality assume that $z$ is the vector
starting with $3p$ $1$s followed by $p$ $0$s. Suppose, for
contradiction, that the degree of $P$ is at most $p-1$ and consider
the polynomial
$ f(x_1,x_2 \ldots ,x_{4p}) =f_1-f_2$
where
$$
f_1= P(x)[1-(\sum_{i=1}^{4p} x_i)^{p-1}]
x_1 x_2 \cdots x_{p+1}(1-x_{3p+1})(1-x_{3p+2}) \cdots (1-x_{4p})
$$
and
$$
f_2=P(z) x_1x_2  \cdots x_{3p}(1-x_{3p+1})(1-x_{3p+2}) \cdots
(1-x_{4p}).
$$
The degree of the polynomial $f_1$ is at most $4p-1$, that of
$f_2$ is exactly $4p$, hence the degree of $f$ is $4p$ and the
coefficient of $\prod_{i=1}^{4p} x_i$ in it is $P(z) \neq 0$.

By the
Combinatorial Nullstellensatz (Theorem \ref{t413}) 
with $F=Z_p$, $n=4p$, $t_i=1$ for all $i$
and $S_i= \{0,1\}$ for all $i$ there is a vector $y=(y_1,y_2, \ldots
,y_{4p}) \in \{0,1\}^{4p}$ so that $f(y_1,y_2, \ldots y_{4p}) \neq
0$. However, the only vector with $\{0,1\}$ coordinates in which
$f_2$ is nonzero is $z$, and as $f_1(z)=f_2(z)=P(z)$, $f(z)=0$.
Thus $y \neq z$ and $f(y)=f_1(y)$. If the Hamming weight of 
$y$ is not divisible by $p$ then the term 
$[1-(\sum_{i=1}^{4p} y_i)^{p-1}]$ vanishes. If the Hamming weight
of $y$ is $2p$ then the term $P(y)$ vanishes. If it is $0$ or $p$, then the
term $y_1y_2 \cdots y_{p+1}=0$ and if it is $4p$ or $3p$ (and $y \neq z$)
then the term $(1-y_{3p+1})(1-y_{3p+2}) \ldots (1-y_{4p})=0$.
Therefore $f(y)=f_1(y)=0$, contradiction. This completes the proof.
$\Box$

\section{High School Coalitions}

In May 2019 Shay Moran showed me a question posted by a woman 
named Ruthi Shaham in a Facebook Group focusing on Mathematics.
She wrote that her son has finished elementary school and was about
to move
to high school. When doing so, each child lists three friends,
and the assignment of children 
into classes ensures that each child will have at least one
of these three friends in his class. Ruthi further wrote that 
her son heard from five of his schoolmates that they found that
they can make their selections in a way that will ensure that all five will
be scheduled to the same class. She tried to check with a paper and pencil
and couldn't decide whether or not this is possible, but she
suspected it is 
impossible.
She thus asked if this is indeed the case, and if so, whether a 
larger group
of children can form such a coalition ensuring they will all 
necessarily be assigned
to the same class. 

In this brief section we show that Ruthi has indeed been right, no 
coalition of five children can ensure they will share the same class.
Moreover, no coalition of any size can ensure to share the same class.
This is related to known problems and results 
in Graph Theory, as are several variants of the problem
mentioned below.

Here is a more formal formulation  of the problem, with general parameters.
Let $N=\{1,2, \ldots ,n\}$ be a finite set of size $n$, 
let $k$ and $r$ be integers,
and suppose $n \geq k+1$. For any collection of subsets $S_i$ of $N$, 
$(1 \leq i \leq n)$,
with $i \not \in S_i$, and $|S_i|=k$ for all $i$, let 
$P(S_1, S_2, \ldots ,S_n)$
be a partition of $N$, so that :
\begin{equation}
\label{e511}
\mbox{For any part~~} N_i \mbox{~~of the partition
and for any~~} j \in N, \mbox{~~if~~} j \in N_i \mbox{~~then~~}
S_j \cap N_i \neq \emptyset.
\end{equation}

Here $N$ denotes the group of children, $S_i$ is the list of friends listed
by child number $i$, 
and the partition of $N$ into parts $N_i$ is the partition
of the set of children into classes. The function $P$ represents the 
way the children are partitioned into classes $N_i$ given their
choices $S_i$, and the condition (\ref{e511}) is the one ensuring that each
child will have at least one other child from his list in his class. 

We say that a subset $R \subset N$
is a successful coalition, if there are choices $S_i, i \in R$ of sets
$S_i$ satisfying $|S_i|=k$ and $i \not \in S_i$ so that for  any sets
$S_j \subset N$ with $|S_j|=k$ for all $j \in N-R$, and for any function
$P$ satisfying the conditions above, all elements of $R$ belong to
the same part of the partition $f(S_1, S_2, \ldots ,S_n)$. Note that 
by symmetry if a successful coalition of size $r$ is possible
then any set of size $r$ can form such a coalition, and hence we may
always assume that $R=\{1,2, \ldots ,r\}$.

The question of Ruthi is whether or not for $k=3$ there can be a successful
coalition $R$ of size $|R|=5$. 
\begin{theo}
\label{t511}

\noindent
\begin{enumerate}
\item
For $k \leq 2$ and every integer $r>1$, every set $R$ of size $r$ can form
a successful coalition.
\item
For any $k \geq 3$ and every $r > 1$ no set of size $r$ can form
a successful
coalition.
\end{enumerate}
\end{theo}

\subsection{Proofs}

Before presenting the general proof, here is a short argument showing
that for $k=3$ no successful coalition of size $5$ is possible. This
proof is a simple application of the probabilistic method.
\vspace{0.2cm}

\noindent
{\bf Claim:}\, Suppose $n \geq 5$, $N=\{1,2, \ldots ,n\}$, 
$R=\{1,2, \ldots ,5\}$, and let $S_1, \ldots ,S_5$ be subsets of
$N$, each of size $3$, so that $i \not \in S_i$ for all
$1 \leq i \leq 5$. Then there are subsets $S_j \subset N$,
for $5 \leq j \leq n$ and there is a partition 
$P(S_1, \ldots ,S_n)$ of $N$ into two disjoint parts $N_1,N_2$ 
satisfying (\ref{e511}) such that $R$ intersects both $N_1$ and $N_2$.
\vspace{0.2cm}

\noindent
{\bf Proof:}\, Color the elements of $N$ randomly red and blue, where each
$i \in N$ randomly and independently is red with probability $1/2$ and
blue with probability $1/2$. The probability that all members of $R$
have the same color is $1/16$. For each fixed $i \leq 5$, the probability
that the color of $i$ is different than that of all elements in $S_i$
is $1/8$. Therefore, with probability at least $1-1/16-5/8>0$ none of these
events happens. Hence there is a coloring in which $R$ contains both
red and blue elements, and every $i \in R$ has at least one member
of $S_i$ with the same color as $i$. Fix such a coloring.
Without loss of generality
$1$ is colored red and $2$ is colored blue.
Let $N_1$ be the set of all
elements colored red and let $N_2$ be the set of all elements colored
blue. For each $j \in N_1-R$ let $S_j$ contain $1$ and for each
$j \in N_2-R$ let $S_j$ contain $2$. It is easy to see that the partition
$N=N_1 \cup N_2$ satisfies (\ref{e511}) but $R$ intersects both
$N_1$ and $N_2$, completing the proof. \hfill $\Box$
\vspace{0.2cm}

\noindent
Note that the above proof does not work for $r \geq 8$, thus the proof
of Theorem \ref{t511} requires a different method, which we show next.
\vspace{0.2cm}

\noindent
{\bf Proof of Theorem \ref{t511}:}\,
The case $k \leq 2$ is very simple. For 
$k=1$ simply define $S_i=\{(i+1)(\bmod~~r)\}$ to see that
the coalition $R=\{1,2,\ldots ,r\}$ is successful.
For $k=2$ and $r=2$, $S_1=\{2,3\}$, $S_2=\{1,3\}$ show that
$\{1,2\}$ is successful.  For 
any larger $r$ add to the above
$S_i=\{1,2\}$ for all $3 \leq i \leq r$.
\vspace{0.2cm}

\noindent
The more interesting part is the proof that for $k\geq 3$ no coalition
of any size $r>1$ can be successful.
The case $r<k$ here is simple. 
One possible proof is to repeat the probabilistic
argument described above for the case $k=3,r=5$. 
Since for $1<r <k$, $k \geq 3$,
$$
\frac{1}{2^{r-1}} + \frac{r}{2^k} \leq \frac{1}{2}+\frac{k-1}{2^{k}}
\leq \frac{1}{2}+\frac{2}{8} <1
$$
the result follows as before. 
(It is also possible to give a direct simple proof
for this case).

For $k \geq 3$, $r \geq k$ consider the digraph whose 
set of vertices is $N$, where for each vertex $i$ and each
$j \in S_i$, $ij$ is a directed edge.  Thus every outdegree in this digraph
is exactly $k$.
Given the sets
$S_1, \ldots ,S_r$ 
of outneighbors of the vertices in $R=\{1,2, \ldots ,r\}$ 
(representing the children attempting
to form a successful coalition), define the sets $S_j$ for $j >r$
in such a way that the induced subgraph on $N-R$ is acyclic.
(For example, we can define $S_j=\{1,2,\ldots ,k\}$ for each
$j>r$, or $S_j=\{j-1, j-2, \ldots, j-k\}$ for each $j > r$. Note that
here we used the fact that $r \geq k$).

The crucial result we use here is a theorem of Thomassen
(\cite{Th}, see also \cite{Al50}
for an extension). This Theorem asserts that any digraph with minimum
outdegree at least $3$ contains two vertex disjoint cycles. Let 
$A$ and $B$ be the sets of vertices of these two cycles. 
Note that both $A$ and $B$ must contain
a vertex of $R$ (as $N-R$ contains no directed cycles).
Let $A', B'$ be two sets of 
vertices satisfying $A \subset A'$, $B \subset B'$ with 
$|A'|+|B'|$ maximum subject to the constraint that
every outdegree in $A'$ is at least $1$ and every outdegree in $B'$ is
at least $1$. We claim that $A' \cup B'$ is the set $N$ of 
all vertices.
Indeed, otherwise, every $v$ in $C=N-(A' \cup B')$ has no outneighbors in 
$A' \cup B'$ (otherwise we could have added it to
either $A'$ or $B'$ contradicting maximality), 
so has at least $k \geq 3>1$
outneighbors in $C$ and then we can replace $A'$ by 
$A' \cup C$ contradicting
maximality. This proves the claim. The assignment to two groups is now 
$N_1=A'$ and $N_2=B'$. Since both $A \subset A'$ and $B \subset B'$ contain
elements of $R$, this shows that $R$ is not a successful coalition,
completing the proof.  \hfill $\Box$
\subsection{Variants}
\begin{enumerate}
\item
What if every child is ensured to have at least 
two of his choices with him in his class ? In this case, even if
$k$ is arbitrarily large (but $r$ is much larger) we do not know
to prove that a
coalition of $r$ cannot ensure they are all in the same group.
This is identical to one of the open questions in  \cite{Al51},
which is the following.
\vspace{0.2cm}

\noindent
{\bf Question:}\, Is there a finite 
positive integer $k$ such that every digraph in
which all oudegrees are (at least) $k$ contains 
two vertex disjoint subgraphs,
each having minimum outdegree at least $2$ ?
\vspace{0.2cm}

\noindent
On the other hand it is easy to see that this is impossible if
$\frac{1}{2^{r-1}}+ \frac{r (1+k)}{2^{k}}<1$. Indeed, if so 
we can split the group of
children randomly into two sets, 
red and blue. With positive probability the 
specific set of $r$ children trying to form a coalition is not
monochromatic, and also for any child in the coalition there are at 
least two of his choices in his group. We can now fix the choices of
all others outside the coalition to ensure they will also be happy with this
partition. It follows that if in this version of the problem
a successful  coalition
of size $r$ is possible, then $r$ has to be at least exponential in
$k$.
\item
Suppose we change the rules, and each child lists
$k$ other children that he does {\em not} like, and wishes 
not to have many of them
in his class. It can then be shown that for any $k$ there is an example
of choices of the children in which each one lists $k$ others he 
prefers to avoid, so that in any partition of the group of children
into $2$ classes, there will always be at least one poor 
child sharing the same class with all the $k$ he listed !  This is based
on another result of Thomassen \cite{Th1}: for every $k$ there is a 
digraph with minimum outdegree $k$ which contains no even directed cycle.
If $D=(N,E)$ 
is such a digraph, and $N=V_1 \cup V_2$ is a partition of its 
vertex set into two disjoint parts, then, as observed in
\cite{Al51}, there is a vertex in one 
of the classes having all its out-neighbors  in the same class.
Indeed, otherwise, starting at an arbitrary vertex $v_1$ we can define
an infinite sequence $v_1,v_2, v_3, \ldots$, where each pair
$(v_i,v_{i+1})$ is a directed edge with one end in $V_1$ and one 
in $V_2$. As the graph is finite, there is a smallest $j$ such 
that there is $i<j$ with $v_i=v_j$, and the cycle
$v_i,v_{i+1},\ldots,v_j=v_i$ is even, contradiction.
On the other hand, by splitting the group of children into
$s \geq 3$  disjoint groups, we  can always ensure that each child
will have in his own class at most $2k/s$ of the $k$ children he 
wants to avoid. This follows from a result of Keith Ball described
in \cite{Al51}.
\end{enumerate}

\section{$\ell_1$-balls and projections of linear codes}

A remarkable known property of the Binomial distribution $Bin(n,p)$ is
that its median is always either the floor or the ceiling of its
expectation $np$. In particular, if the expectation is an integer
then this is also the median. The following more general result 
is proved by
Jogdeo and Samuels in \cite{JS}.
\begin{theo}[\cite{JS}, Theorem 3.2 and Corollary 3.1]
\label{t611}
Let $X=X_1+X_2+ \ldots +X_n$ be a sum of independent indicator random
variables where for each $i$, $Pr(X_i=1)=p_i$ and
$Pr(X_i=0)=1-p_i$. Then the median of $X$ is always the floor or
the ceiling of its expectation $\sum_{i=1}^n p_i$.
\end{theo}
This theorem can be used to derive several interesting results.
Here we describe one quick application and another more complicated
one in which it is convenient (though not absolutely necessary) to
use it, combined with several additional ingredients.

\subsection{ $\ell_1$-balls and Hamming balls in the discrete cube}

If $n$ is even, $d=n/2$ and $x=(1/2,1/2, \ldots ,1/2)$ is the
center of the $n$-dimensional real unit cube $[0,1]^n$, then the
$\ell_1$-ball of radius $d$ centered at $x$ contains all the
$2^n$ 
points of the discrete cube $\{0,1\}^n$. On the other hand, any
Hamming ball of radius $d$ centered at a vertex $y$ of this
discrete cube
contains only $\sum_{i=0}^d  {n \choose i} =(\frac{1}{2}+o(1))
2^n$ points of the cube, where the $o(1)$-term tends to $0$ as $n$
tends to infinity. Madhu Sudan \cite{Su} asked me whether a similar
bound holds for any $\ell_1$-ball of integral radius. 
The precise statement of the question is as follows:
\vspace{0.1cm}

\noindent
Is it true that for any positive integer  $d$ and for any $\ell_1$-ball
$B$
(centered at any real point in $R^n$) there is a Hamming ball
of the same radius $d$ centered at a point in  $\{0,1\}^n$ that 
contains at least half the points in $B \cap \{0,1\}^n$ ?

The following stronger result shows that this is indeed the case.
\begin{theo}
\label{t621}
For any real $x=(x_1,x_2, \ldots ,x_n)$ in $R^n$ and 
for any subset $A$ of points of
$B(x,d) \cap \{0,1\}^n$, where $B(x,d)$ is the $\ell_1$-ball of
radius $d$
centered
at $x$, and $d$ is an integer, there is $y \in \{0,1\}^n$ so that
$|A \cap B(y,d)| \geq |A|/2.$
\end{theo}

\begin{proof}
Note, first, that we may assume that $x_i \in [0,1]$ for all $i$.
Indeed, otherwise, replace $x_i$ by $1$ if $x_i>1$ and by $0$ if 
$x_i<0$. This modification only decreases the  $\ell_1$-distance
between $x$ and any point in $\{0,1\}^n$. Therefore $A$ is a subset of
the ball $B(x,d)$ for the modified vector $x$ too. We thus may and will
assume that $x \in [0,1]^n$. 
Let $y=(y_1,y_2, \ldots ,y_n)$ be a random binary vector
obtained by choosing, for each i, randomly and
independently, $y_i$ to be $1$ with probability $x_i$ and 
$0$ with probability
$(1-x_i)$. For each point $a \in A$, the $\ell_1$-distance 
between $y$ and $a$ is a
random variable which is a sum of independent Bernoulli random
variables  and its
expectation
is exactly the $\ell_1$ distance between $a$ and $x$, which is at most 
$d$.
By Theorem \ref{t611} of Jogedo and Samuels stated above
the probability that this random
variable is at most $d$ is at least a half. It follows by linearity of
expectation that the expected number of
points of $A$ within distance at most $d$ from $y$ is at least
$|A|/2$, and
thus there is a $y$ as needed.
\end{proof}

\subsection{Random projections of linear codes}

Let $F$ be a finite or infinite field, and let $V$ be a linear code
of length $n$, dimension $k$ and minimum relative distance 
at least $\delta$
over $F$. Thus $V$ is a subspace of dimension $k$ of $F^n$, and the
number of nonzero coordinates of any nonzero codeword $v \in V$
is at least $\delta  n$. Let $m$ be an integer. A 
projection of $V$ on $m$ random coordinates is obtained by selecting a
random (multi)set $I$ of $m$ coordinates of $[n]$, chosen with
repetitions. With this random choice of $I$ let $V_m \subset F^m$ 
be the vector space over $F$ consisting  of all
vectors $\{ (v_i)_{i \in I}~: ~ v=(v_1,v_2, \ldots ,v_n) \in V\}$.
One may expect that if $m$ is large, then typically the
vector space $V_m$, considered as a linear code of length $m$ over
$F$, will have dimension $k$ and minimum distance not much smaller
than $\delta m$. This is easy to prove by a standard application of 
Chernoff's Inequality and the union bound, provided
$m$ is sufficiently large as a function of $|F|,k$ and
$\delta$. It is, however, not clear at all 
that this is the case for $m$ of size
independent of the size of the field $F$ (which may even be
infinite).  Such a statement is proved by Saraf and Yekhanin in
\cite{SY}.
\begin{theo}[\cite{SY}, Theorem 3]
\label{t631}
Let $V$ be a linear code of dimension $k$, length $n$ and minimum
relative distance $\delta$ over an arbitrary field $F$.
If $m$ is at least $c(\delta) k$ and $V_m$ is a projection of
$V$ on $m$ random coordinates
then with probability at least
$1-e^{-\Omega(\delta m)}$ the dimension of $V_m$ is $k$ and its
minimum distance is at least $\delta m/8$.
\end{theo}
One can check that the estimate the proof in \cite{SY} provides for
$c(\delta)$ is $b \frac{\log (1/\delta)}{\delta}$ for a
sufficiently large absolute constant $b$. Note, however, that the
minimum relative distance obtained is only $\delta /8$, this loss in the
minimum relative distance is inherent in the approach of \cite{SY}.

Here we show how to apply some of the techniques in the study of
$\eps$-nets and $\eps$-approximations in range spaces with
finite Vapnik-Chervonenkis dimension to get an improved  version of
the above theorem in which the relative minimum distance obtained
can be arbitrarily close to $\delta$. 
\begin{theo}
\label{t632}
There exists an absolute positive constant $c$ so that the
following holds. Let $B>2$ be an integer, and let
$V$ be a linear code of dimension $k$, length $n$ and minimum
relative distance $\delta$ over an arbitrary field $F$.
If $m$ is at least 
$c\frac{B^2 k}{\delta} \log (B/\delta)$
and $V_m$ is a projection of
$V$ on $m$ random coordinates
then with probability at least
$1-e^{-\Omega(\delta m/B^2)}$ the dimension of $V_m$ is $k$ and
its minimum distance 
is at least $(\frac{B-1}{B+1})\delta m$.
\end{theo}
Taking $B$ to be a large fixed constant  we get that typically the
minimum relative distance of $V_m$ is close to  $\delta$, and the
estimates for $m$ and for the failure probability are essentially
as in Theorem \ref{t631}. 

We start with a quick reminder of the relevant facts about VC-dimension.
The {\em Vapnik-Chervonenkis dimension} $VC(\CC)$ of a (finite) 
family of binary vectors 
$\CC$ is the maximum cardinality of a set of coordinates $I$ such
that for every binary vector $(b_i)_{i \in I}$ there is a  $C \in
\CC$ so that $C_i=b_i$ for all $i \in I$.  (In this case we say
that the set $I$ is {\em shattered} by $\CC$).
Suppose the vectors
in the family are of length $n$. An {\em $\eps$-net} for the family is
a subset $I \subset [n]$ such that for every $C \in \CC$ of Hamming
weight at least $\eps n$ there is an $i \in I$ so that
$C_i=1$. An $\eps$-approximation for the family is
a sub(multi)set $I \in [n]$ so that for every $C \in \CC$
$$
| ~\frac{|\sum_{i=1}^n C_i|}{n}-
\frac{|\sum_{i \in I}^n C_i|}{|I|} ~| < \eps.
$$
A basic result proved by Vapnik and Chervonenkis \cite{VC} (with
a logarithmic improvement by Talagrand \cite{Ta}), is that if 
$VC(\CC) \leq d$ then a random set of $\Theta(\frac{d}{\eps^2})$
coordinates is typically an
$\eps$-approximation.
A similar result, proved by Haussler and
Welzl \cite{HW}, is that for such a  $\CC$ a random set of
$\Theta(\frac{d}{\eps} \log (1/\eps))$ coordinates is typically an
$\eps$-net. 
Another basic combinatorial result is the Sauer-Perles-Shelah
Lemma: if $VC(\CC) \leq d$ then the number of distinct projections of 
the set of vectors in $\CC$ on any set of $t$ coordinates is at
most $g(d,t)=\sum_{i=0}^d {t \choose i}.$ 

The relevance of the VC-dimension to projections of linear codes
is the following simple observation.
\begin{claim}
\label{cl633}
Let $F$ be an arbitrary field, and 
let $V \subset F^n$ be a linear subspace of dimension $k$ over $F$.
For each vector $v \in C$ let $C=C(v)$ denote the indicator vector
of the support of $v$, that is, $C_i=1$ if $v_i \neq 0$ and
$C_i=0$ is $v_i=0$. Put $\CC=\{C(v): v \in V\}$. Then $VC(\CC)
= k$. 
\end{claim}
\begin{proof}
Since the dimension of $V$ is $k$ it contains a set of $k$ vectors 
$v^{(i)}$ 
such that there is a set $I=\{i_1, i_2, \ldots ,i_k\}$ 
of $k$ coordinates so that $v^{(i)}_{i_j}$ is $1$ for  $i=j$ and
$0$ otherwise. The supports of the set of all linear combinations 
with $\{0,1\}$-coefficients 
of these vectors shatter the set $I$, implying that
$VC(\CC) \geq k$. Conversely, if there is a set of coordinates $J$
shattered by the vectors in $\CC$, then for each $j \in J$ there is
a vector in $V$ with $v_j \neq 0$ and $v_i=0$ for all $i \in J-j$.
These $|J|$ vectors are clearly linearly independent, implying that
$|J| \leq k$ and completing the proof.
\end{proof}

The above claim and the known result stated above about
$\eps$-approximation for families of vectors with finite
$VC$-dimension suffice to prove a version of Theorem \ref{t632}
with $m=\Theta(\frac{B^2 k}{\delta^2})$. Indeed, we simply consider
a $2\delta/(B+1)$-approximation for the set $\CC$ corresponding
to $V$. Similarly, the result about $\eps$-nets shows that
typically the dimension of $V_m$ is $m$.

In order to prove the improved estimate for $m$ stated in
the theorem we show that in the setting
here the bound can be improved to be closer to that in the
theorem about $\delta$-nets. This is proved in the following
result, which applies to general collections of vectors with 
a bounded VC-dimension.
\begin{prop}
\label{p634}
There exists an absolute positive constant $c>1$ such that the
following holds. 
Let $\CC$ be a family of binary vectors of length $n$, and assume
that $VC(\CC) \leq d$. Let $X$ be a random multiset of $m$
coordinates, with $m=c \frac{B^2 d }{\eps} \log (B/\eps)$, where
$B>2$ is an integer. Then
with probability at least $1-e^{-\Omega(\eps m/B^2)}$, for every
$C \in \CC$ satisfying $\sum_{i=1}^n C_i \geq \eps n$
we have $\sum_{i \in X} C_i \geq \frac{B-1}{B+1}\eps m.$
\end{prop}
In order to prove the above statement, we need some standard
estimates for large deviations of the hypergeometric distribution.
The estimate we use here was first proved by  Hoeffding \cite{Hoe},
see also \cite{JLR}, Theorem 2.10 and Theorem 2.1.
\begin{lemma}[Hoeffding \cite{Hoe}, see also \cite{JLR}]
\label{l635}
Let $H$ be the hypergeometric  distribution given by the cardinality
$|R \cap S|$ where $S$ is a random subset of cardinality $m$ in 
a set of size $N$ containing a subset $R$  of cardinality 
$pN$. Then the probability that $H$ is smaller than 
$pm-t$ is at most $e^{-t^2/2pm}$.
\end{lemma}
\vspace{0.1cm}

\noindent
{\bf Proof of Proposition \ref{p634}:}\, 
Let $m$ be as in the statement of the proposition and let
$X=(x_1, \ldots ,x_m)$ be a random multiset obtained
by $m$ independent random choices, with repetitions, of elements of $[n]$.
For  $C \in \CC$ we let $|C|$ denote $\sum_{i=1}^n C_i$ and let
$|C \cap X|$ denote $|\{i: C_{x_i}=1\}|.$
Let $E_1$ be the
following event:
$$
E_1=\{ \exists C \in \CC: |C| \geq \eps n, |C \cap X | < 
\frac{B-1}{B+1}\eps m \}
$$
To complete the proof we have to show that the probability of $E_1$ is
as small as stated in the proposition.
To do so, we make an additional random choice
and define another event as follows. Independently of the previous
choice, let $T=(y_1, \ldots ,y_{Bm})$ be obtained by $Bm$
independent
random choices of elements of $[n]$.
Let $E_2$ be the event defined by
$$
E_2=\left\{ \exists C \in \CC: |C| \geq \eps n,
|C \cap X| < \frac{B-1}{B+1}\eps m,
|C \cap T| \geq \lfloor B \eps m \rfloor \right\}
$$

\begin{claim} 
\label{c636}
$Pr(E_2) \geq \frac{1}{2} Pr{E_1}$.
\end{claim}
\begin{proof}
It suffices to prove that the conditional probability $Pr(E_2 |
E_1)$ is at least $1/2$. Suppose that the event $E_1$ occurs. Then
there is a $C \in \CC$ such that $|C| \geq \eps n $ and $|C
\cap X| < \frac{B-1}{B+1} \eps m$. 
The conditional probability above is clearly
at
least the probability that for this specific $C$, $|C \cap T| \geq
\lfloor B \eps m \rfloor $. However $|C \cap T|$ is a binomial random
variable with
expectation at least $ B \eps m$,  and therefore, by Theorem
\ref{t611} its median is at least the floor of that, implying the
desired result. 
\end{proof}
\begin{claim} 
\label{c637}
$$
Pr(E_2) \leq g(d,(B+1)m) 2^{-\epsilon m/8(B+1)^2 }
$$
\end{claim}
\begin{proof}
The random choice of $X$ and $T$ can be described in the following
way, which is equivalent to the previous one. First choose $X
\cup T = ( z_1, \ldots ,z_{(B+1)m})$ by making $(B+1)m$ random independent
choices of elements of $[n]$ (with repetitions), 
and then choose randomly precisely
$m$ of the elements $z_i$ to be the set $X$, where the remaining elements
$z_j$ form the set $T$. For each member $C \in \CC$
satisfying $|C | \geq \eps n $, let $E_C$ be the event that
$$
| C \cap T | \geq \lfloor B \eps m  \rfloor~~ \mbox{and}~~ |C \cap X| < 
\frac{B-1}{B+1}\eps m.
$$
A crucial fact is that if $C,C' \in \CC$ are two ranges, $|C | 
\geq \eps n$
and $|C'| \geq \eps n$ and if $C \cap (X \cup T)=
C' \cap (X \cup T )$, then the two events $E_C$ and $E_{C'}$,
when both are conditioned on the choice of $X \cup T$, are
identical. This is because the occurrence of $E_C$ depends only on
the intersection $C \cap (X \cup T) $. Therefore, for any fixed
choice of $X \cup T$, the number of distinct events $E_C$ does not
exceed the number of different sets in the projection 
of $\CC$  on the coordinates ${X \cup T}$.
Since the VC-dimension is at most $d$,
this number does not exceed $g(d,(B+1)m)$, by the
Sauer-Perles-Shelah Lemma.

Let us now estimate the probability of a fixed event of the form
$E_C$, given the choice of $X \cup T$.
This probability is at most the probability that a hypergeometric
random variable counting the size of the intersection of
a random set of $m$ elements with a subset $R$ of size at least
$\lfloor B \eps m \rfloor$ in a set of size $N=(B+1)m$ is
smaller than $\frac{B-1}{B+1} \epsilon m$. By Lemma \ref{l635}, 
and using the fact that the choice of $m$ implies that
$\lfloor B \eps m \rfloor> (B-1/2) \eps m$ this probability is
smaller than $e^{-\eps m/8(B+1)^2}$. 
\end{proof}

By Claims~\ref{c636} and~\ref{c637}, $Pr(E_1) \leq 2g(d,(B+1)m)
2^{-\eps m/8(B+1)^2}$. The assertion of the theorem follows
using the fact that 
$$
g(d,(B+1)m)< (\frac{2e(B+1)m}{d})^d.
$$
\hfill  $\Box$
\vspace{0.1cm}

\noindent
{\bf Proof of Theorem \ref{t632}:}\, 
Let $V$ be a linear code of length $n$, dimension $k$ and minimum
relative distance $\delta$. Let $\CC$ be the set of all indicator
vectors of supports of vectors in $V$. By Claim \ref{cl633} the
VC-dimension of $\CC$ is at most $k$, and by definition the Hamming
weight of each member $C$ of $\CC$ is at least $\delta n$. The
desired result thus follows from Proposition \ref{p634}.
\hfill $\Box$

\section{Connected dominating sets}
The first result in this section was obtained in joint discussions 
with Michael Krivelevich \cite{Kr}.

Let $G=(V,E)$ be a connected graph. Let $\gamma(G)$ denote the
minimum  size of a dominating set in it, that is, the minimum
cardinality of a set of vertices $X \subset V$ so that each $v \in
V-X$ has at least one neighbor in $X$. Let
$\gamma_c(G)$
denote the minimum size of a connected dominating set of  $G$, that
is, the minimum cardinality of a dominating set of vertices $X$ so
that the induced subgraph of $G$ on $X$ is connected. 
One of the reasons this parameter has been studied extensively
is the fact that $|V|-\gamma_c(G)$ is exactly the maximum possible
number of leaves in a spanning tree of $G$.
It is well known that if the minimum degree in $G$ is $k$ and its
number of vertices is $n$, then
$\gamma(G) \leq \frac{n (\ln (k+1)+1)}{k+1}$. See \cite{Lo} or
\cite{AS}, Theorem 1.2.2 for a proof. As mentioned in \cite{AS}
this is asymptotically tight for large $k$, see, e.g., \cite{AW}
for a proof that for any $\eps>0$  and $k>k_0(\eps)$ a random
$k$-regular graph on $n$ vertices is unlikely to contain a
dominating set of size at most $(1-\eps) \frac{n \ln k}{k}$.

Caro, West and Yuster \cite{CWY} proved that for every 
connected graph $G$ with $n$ vertices and 
minimum degree $k$, $\gamma_c(G)$ is also 
not much larger than $\frac{n \ln (k+1)}{k+1}$. The precise
statement of their result is as follows.
\begin{theo}[\cite{CWY}]
\label{t711}
Let $G$ be a connected graph with $n$ vertices and minimum degree
at least $k$. Then
$$
\gamma_c(G) \leq \frac{n(\ln(k+1)+0.5\sqrt{ \ln (k+1)} +145)}{k+1}
$$
\end{theo}
Here we first prove a similar result with a slightly better estimate. 
\begin{theo}
\label{t712}
Let $G$ be a connected graph with $n$ vertices and minimum degree
at least $k$. Then
$$
\gamma_c(G) \leq \frac{n(\ln(k+1)+\ln \lceil \ln (k+1) \rceil+4)}{k+1}.
$$
\end{theo}
The main merit here is not the improved estimate, but
the proof, which is much simpler than the one in \cite{CWY}. Like
the proof in \cite{CWY}, it 
provides a simple efficient algorithm for finding a connected
dominating set of the required size for a given input graph.
As a byproduct of the proof we get an upper bound for
the difference between $\gamma(G)$ and $\gamma_c(G)$, as stated in
the following theorem.

Define a function $f=f_{n,k}$ mapping $[1,\infty)$ to
$[0,\infty)$ 
as follows.
For any real $x \geq 1$, let
$x=(y+z) \frac{n}{k+1}$ with $y \geq 0$ an integer 
and $z \in [0,1]$ a real:
\begin{enumerate}
\item
If $y=0$ then $ f(x)=\frac{n}{k+1}2z-2.$
\item
If $ y =1$ then
$ f(x)= \frac{n}{k+1}(\frac{z}{y}+ 2)-2.$
\item
If $y \geq 2$ then $ f(x)= \frac{n}{k+1}
( \frac{z}{y}+\frac{1}{y-1}+ \cdots + \frac{1}{1}+2)-2. $
\end{enumerate}
The function $f$ is piecewise linear and monotone increasing. Its
derivative, which exists in all points of $(1,\infty)$ besides the
integral multiples of $\frac{n}{k+1}$, is (weakly) decreasing, 
thus $f$ is concave. In addition it satisfies the following.
For every $x = (w+z) \frac{n}{k+1} > \frac{n}{k+1}$ with
$w \geq 1$ an integer and $z \in [0,1]$ a real, and for every $w'$
satisfying $w \leq w' \leq x-1$
\begin{equation}
\label{e7new}
f(x) \geq f(x-w')+1
\end{equation}
Indeed, the derivative of $f(z)$  is at least $\frac{1}{w}$ for every
$z$ in $(x-w',x]$ (besides the integral multiples of
$\frac{n}{k+1}$), and thus $f(x)-f(x-w')$, which is the integral of
this derivative from $x-w'$ to $x$, is at least $w' \cdot
\frac{1}{w} \geq 1$.
\begin{theo}
\label{t713}
Let $G$ be a connected graph with $n$ vertices, minimum degree 
at least $k$ and domination number  $\gamma=\gamma(G)$. 
Then $ \gamma_c(G) \leq \gamma + f_{n,k}(\gamma). $
Therefore 
$$ 
\gamma_c(G) < \gamma+ \frac{n}{k+1} (\ln \lceil \ln (k+1) \rceil +3).
$$ 
\end{theo}
We also describe an improved argument that provides a better estimate 
than the ones in Theorems \ref{t711}, \ref{t712}.
\begin{theo}
\label{t714}
Let $G$ be a connected graph with $n$ vertices and minimum degree
at least $k$. Then 
$$
\gamma_c(G) \leq  \frac{n}{k+1}(\ln (k+1)+4)-2.
$$
\end{theo}
The proof here  too provides an efficient randomized algorithm 
for finding a connected dominating set with expected size as in
the theorem. This algorithm can be derandomized and converted into
an efficient deterministic algorithm.

\subsection{Proofs}

In the proofs we use the following  simple lemma.
\begin{lemma}
\label{l715}
Let $G=(V,E)$ be a connected graph with $n$ vertices and minimum degree at
least $k$. Let $S \subset V$ be a dominating set of $G$, let $H$
be the induced subgraph of $G$ on $S$, and suppose the number of
its connected components is $x=(y+z)\frac{n}{k+1}$ where
$ y$ is a nonnegative integer and $0 \leq z  \leq 1$ is a real. 
Then $\gamma_c(G) \leq |S|+f(x)$, where $f=f_{n,k}$ is the function
defined in the previous subsection.
\end{lemma}
\vspace{0.2cm}

\noindent
{\bf Proof:}\,
Starting with the dominating set $S$ we prove, by
induction on $x$, that it is always possible to add to it at most
$f(x)$ additional vertices to get a connected dominating set.
For $x=1$ the given set is already connected, and as $f(1)=0$ the
result in this case is trivial.
If $1<x \leq \frac{n}{k+1}$ 
we note that as long as there are at least two components, each one
$C$ can be merged to another one by adding at most two vertices. Indeed,
every vertex in the second neighborhood of $C$ is dominated, hence
adding the two vertices of a path from $C$ to any such vertex
merges $C$ to another component. This means that by adding at most
$2 (x-1)=f(x)$  vertices to $S$ we get a connected
dominating set, as needed.

If $x> \frac{n}{k+1}$ pick arbitrarily one vertex $v=v(C)$ in each
of the $x$ connected components of $H$ and let $N(v)$ denote
its closed neighborhood consisting of $v$ and all its neighbors in
$G$. This set is of size at least $k+1$. Therefore there is a
vertex $u$ of $G$ that belongs to at least $\lceil (k+1)x/n \rceil$
of these closed neighborhoods. (This can in fact be slightly improved
as none of the vertices of the dominating set belongs to more than one
such closed neighborhood, but we do not use this improvement
here).
Define $S'=S \cup \{u\}$ and note
that adding $u$ merges at least $\lceil (k+1)x/n \rceil$ components. 
Therefore, if $x > w \frac{n}{k+1}$  for an integer $w \geq 1$,
then the number of connected components of the induced subgraph of
$G$ on the dominating set $S'$ is $x-w'$ for some $w' \geq w$.
By induction one can add to $S'$ at most $f(x-w')$ additional
vertices to get a connected dominating set, and the desired result
follows from (\ref{e7new}).
\hfill $\Box$

The proof clearly supplies an efficient deterministic
algorithm for finding a connected dominating set of the required
size, given the initial dominating set $S$.
\vspace{0.2cm}

\noindent
{\bf Proof of Theorem \ref{t713}:}\,
This is an immediate consequence of Lemma \ref{l715} together with
the obvious fact that if $\gamma(G)=\gamma$ then $G$ contains a
dominating set $S$ of size $\gamma$ with at most $|S|=\gamma$
connected components.  The 
known fact that $\gamma\leq \frac{n}{k+1}(
\ln (k+1)+1)$ implies that $\gamma \leq \frac{n}{k+1} (y+z)$
with $y=\lceil \ln(k+1) \rceil$ and $z=1$. The definition of the
function $f=f_{n,k}$ thus implies that
$$
f_{n,k}(\gamma) \leq \frac{n}{k+1}(\frac{1}{y}+\frac{1}{y-1}+
\ldots + \frac{1}{1}+2) -2 < \frac{n}{k+1}(\ln y+3),
$$
completing the proof.
\hfill $\Box$
\vspace{0.2cm}

\noindent
{\bf Proof of Theorem \ref{t712}:}\,
This follows from Theorem \ref{t713} together with the fact that
$\gamma(G) \leq \frac{n}{k+1}  (\ln (k+1)+1)$. 
\hfill $\Box$

In order to prove Theorem \ref{t714} we need two simple lemmas.
The first one is a known fact, c.f., e.g., \cite{CS}, Formula (3.2). 
for completeness we include a short proof.
\begin{lemma}
\label{l716}
For a positive integer $k$ and a real $p \in (0,1)$, let
$B(k,p)$ denote the Binomial random variable with parameters
$k$ and $p$. Then the expectation of $\frac{1}{B(k,p)+1}$ satisfies
$$
E[\frac{1}{B(k,p)+1}]=\frac{1}{(k+1)p}-\frac{(1-p)^{k+1}}{(k+1)p}.
$$
\end{lemma}
\vspace{0.2cm}

\noindent
{\bf Proof:}\,  By definition
$$
E[\frac{1}{B(k,p)+1}]
=\sum_{i=0}^k \frac{1}{i+1} {k \choose i} p^i (1-p)^{k-i}
=(1-p)^k \sum_{i=0}^k \frac{1}{i+1}{k \choose i} (\frac{p}{1-p})^i.
$$
By the Binomial formula $(1+x)^k = \sum_{i=0}^k {k \choose i} x^i$.
Integrating we get
$$
\frac{(1+x)^{k+1} -1}{k+1}= \sum_{i=0}^k \frac{1}{i+1} {k \choose
i} x^{i+1}.
$$
Dividing by $x$ and plugging $x=\frac{p}{1-p}$ the desired result
follows.
\hfill $\Box$

\begin{lemma}
\label{l717}
Let $H=(V,E)$ be a graph. For every $v \in V$ 
let $d_H(v)$ denote the degree of $v$ in $H$. Then the number of
connected components of $H$ is at most $D(H)=\sum_{v \in V}
\frac{1}{d_H(v)+1}$.
\end{lemma}
\vspace{0.2cm}

\noindent
{\bf Proof:}\, The contribution to $D(H)$ from the vertices in any
connected component $C$ of $H$ with $m$ vertices is 
$$
\sum_{v \in C} \frac{1}{d(v)+1} \geq \sum_{v \in C}
\frac{1}{m} =1.
$$
\hfill $\Box$
\vspace{0.2cm}

\noindent
{\bf Proof of Theorem \ref{t714}:}\, 
Recall that the function $f=f_{n,k}$  defined in the previous
subsection is concave.
Therefore, by Jensen's Inequality, for every positive random
variable $X$, $E[f(X)] \leq f(E[X])$.

Let $G=(V,E)$ be a connected graph with $n$ vertices and
minimum degree at least $k$.
By Lemma \ref{l715} if there is a dominating set $S$ of $G$
and the induced subgraph of $G$ on $S$ has
$x$ connected components, then 
\begin{equation}
\label{e711}
\gamma_c(G) \leq |S|+f(x).
\end{equation}
For a dominating set $S$, let $H=H(S)$ be the induced subgraph
of $G$ on $S$, and put $D(H)=\sum_{v \in S} \frac{1}{d_H(v)+1}$
where $d_H(v)$ is the degree of $v$ in $H$. 
By Lemma \ref{l717} the number of connected components of
$H$ is at most $D(H)$, and since the function $f=f_{n,k}$ defined
above is monotone increasing this implies, by (\ref{e711}), that
\begin{equation}
\label{e712}
\gamma_c(G) \leq |S| +f(D(H)) =|S|+f (\sum_{v \in S}
\frac{1}{d_H(v)+1}).
\end{equation}
We next describe a random procedure for generating a dominating set
$S$ and complete the proof by upper bounding the expectation of the
right-hand-side of (\ref{e712}). The procedure is the standard one
described in \cite{AS}, Theorem 1.1.2 for generating a dominating
set. Define $p=\frac{\ln(k+1)}{k+1}$ and let $T$ be a random set of 
vertices of $G$ obtained by picking, randomly and independently,
each vertex of $G$ to be a member of $T$ with probability $p$.
Let $Y=Y_T$ be the set of all vertices of $G$ that are not
dominated by $T$, that is, all vertices in $V-T$ that have no
neighbors in $T$. The set $S$ defined by $S=T \cup Y_T$ is clearly
dominating. The expected size of $T$ is $np$. The expected size of
$Y_T$ is at most $n(1-p)^{k+1}$, since for any vertex $v$ the
probability it lies in $Y_T$ is exactly $(1-p)^{d_G(v)+1}
\leq (1-p)^{k+1}$, and the bound for the expectation of 
$|Y_T|$ follows by linearity
of expectation.  We proceed to bound the expectation of
$f(\sum_{v \in S} \frac{1}{d_H(v)+1}).$ By Jensen's Inequality 
and the convexity of $f$ mentioned above this
is at most $f(E[\sum_{v \in S} \frac{1}{d_H(v)+1}]).$ 
Since $f$ is monotone increasing it suffices to bound the 
expectation $E[\sum_{v \in S} \frac{1}{d_H(v)+1}]$.

Fix a vertex $v$. The probability it belongs to
$Y_T$ (and hence has degree $0$ in $H$) is 
$(1-p)^{d+1}$, where $d$ is its degree in $G$. 
The probability it belongs to $T$
and has degree $i$ in $H$ is  
$p {d \choose i} p^i (1-p)^{d-i}$.
Therefore, the expectation of $\frac{1}{d_H(v)+1}$
is, by Lemma \ref{l716},
$$
(1-p)^{d+1} + p(\frac{1}{(d+1)p} -\frac{(1-p)^{d+1}}{(d+1)p})
< (1-p)^{k+1}+\frac{1}{k+1}.
$$
Since $(1-p)^{k+1} \leq e^{-p(k+1)}=\frac{1}{k+1}$ 
this implies, by linearity
of expectation, that
$$
E[\sum_{v \in S} \frac{1}{d_H(v)+1}] \leq \frac{2n}{k+1}.
$$
Using, again, linearity of expectation and the fact that
$f_{n,k}(\frac{2n}{k+1}) =3 \frac{n}{k+1}-2$
we conclude that the expectation of the right-hand-side of
(\ref{e712}) is at most
$$
np+n(1-p)^{k+1}+ 3\frac{n}{k+1}-2 \leq \frac{n}{k+1} (\ln (k+1)+4)-2.
$$
Therefore there is a dominating set $S$ for which this expression
is at most the above quantity,
completing the proof. 
\hfill $\Box$

\subsection{Algorithm}

\noindent
The proof of Theorem \ref{t714} clearly supplies a randomized
algorithms generating a connected dominating set of expected size
at most as in the theorem in any given connected input graph
$G=(V,E)$ with
$n$ vertices and minimum degree at least $k$. This algorithm can 
be derandomized using the method of conditional expectations,
yielding a polynomial time deterministic algorithm for finding 
such a connected dominating set. Here is the
argument. Let $v_1,v_2, \ldots ,v_n$ be an arbitrary numbering of
the vertices of $G$. The algorithm generates a dominating set
$S$ satisfying 
$$
|S| +f(D(H)) =|T|+|Y_T| +f (\sum_{v \in S} \frac{1}{d_H(v)+1}  )
\leq \frac{n}{k+1} (\ln (k+1)+4)-2,
$$
where $f=f_{n,k}$ is the function defined in the proof of Theorem
\ref{t714}, $H$ is the induced subgraph of $G$ on $S=T \cup Y_T$ and
$D(H)=\sum_{v \in S} \frac{1}{d_H(v)+1}$.
Once such an $S$ is found it is clear that the proof of
the theorem provides an efficient way to construct a connected
dominating set of the required size using it.

The algorithm produces $S$ as above by
going over the vertices $v_i$ in order, where in
step $i$ the algorithm decides whether or not to add $v_i$ to $S$.
Let $S_i$ denote $S \cap \{v_1,v_2, \ldots ,v_i\}$. Thus 
$S_0=\emptyset$. For each $i$, $0 \leq i \leq n$, define
a potential function $\psi_i$ in terms of the conditional expectations
of $|S|=|T|+|Y_T|$ given $S_i$, which is denoted by
$E[|S||S_i]$ and the conditional expectation 
of $\sum_{v \in S} \frac{1}{d_H(v)+1}$ given $S_i$, denoted by 
$E[\sum_{v \in S} \frac{1}{d_H(v)+1}| S_i]$. In this notation
$$
\psi_i=E[|S||S_i]+f(E[D(H) |S_i]= E[|T| | S_i]+E[|Y_T| | S_i]
+f(E[\sum_{v \in S} \frac{1}{d_H(v)+1}| S_i]).
$$
Given the graph  $G$ and the set $S_i$, it is not difficult to
compute $\psi_i$ in polynomial time. 
Indeed, by linearity of expectation, the conditional expectation
$E[|T||S_i]$ is computed by adding the contribution of each vertex
$v=v_j$ to it. For $j \leq i$ this contribution is $1$ if
$v_j \in T$ and $0$ if $v_j \not \in T$. For $j>i$ the contribution
is $p$. The contribution of $v_j$ to $E[Y_T|S_i]$ 
is $0$ if
$v_j$  is already dominated by a vertex in $S_i$, and if it is not,
then it is 
$(1-p)^s$, where $s$ is the number of neighbors of $v_j$ 
(including $v_j$ itself if $j>i$)
in the set $V-\{v_1,v_2, \ldots ,v_i\}$.  

The conditional expectation $E[\sum_{v \in S} \frac{1}{d_H(v)+1}|
S_i]$ is also computed using linearity of expectation, where the
contribution of each vertex  $v_j$ is 
$E[\frac{1}{d_H(v_j+1)}| S_i]$.  This is also simple to compute
in all cases. We describe here only one representative example.
If $j>i$, $q$ of the neighbors of $v_j$ appear in  $S_i$,
and the number of its neighbors in $G$ which lie in
$V-\{v_1, v_2, \ldots ,v_i\}$ is $s$, then 
$$
E[\frac{1}{d_H(v_j+1)}| S_i]=p \cdot \sum_{a=0}^{s} {s \choose a}
p^a (1-p)^{s-a}\frac{1}{q+1+a}.
$$
A similar expression exists in every other possible case.

Put $\psi_i=\psi_i^{(T)}+\psi_i^{(Y)}+
\psi_i^{(f)}$, where
$\psi_i{(T)}=E[|T| | S_i]$,
$\psi_i{(Y)}=E[|Y_T| | S_i]$,
 and 
$\psi_i^{(f)}=f[E(D(H)| S_i]$.
By the definition of conditional expectation
\begin{equation}
\label{e713}
\psi_{i}^{(T)} =p E[|T|~ |~ S_{i+1}=S_i \cup v_{i+1}]
+(1-p) E[|T|~ |~ S_{i+1}=S_i]
\end{equation}
and 
\begin{equation}
\label{e7131}
\psi_{i}^{(Y)} =p E[|Y_T|~ |~ S_{i+1}=S_i \cup v_{i+1}]
+(1-p) E[|Y_T|~ |~ S_{i+1}=S_i]
\end{equation}
Similarly, using the fact that the function $f$ is concave
$$
\psi_i^{(f)}=
f(p E[\sum_{v \in H} \frac{1}{d_H(v_j)+1)}| S_{i+1}=S_i \cup v_{i+1}]
+(1-p) E[\sum_{v \in H} \frac{1}{d_H(v_j)+1)}| S_{i+1}=S_i] )
$$
$$
\geq p f(E[\sum_{v \in H} \frac{1}{d_H(v_j)+1)}| S_{i+1}=S_i \cup
v_{i+1}])
+(1-p) f ( E[\sum_{v \in H} \frac{1}{d_H(v_j)+1})| S_{i+1}=S_i] )
$$
$$
\geq \min \{ 
f(E[\sum_{v \in H} \frac{1}{d_H(v_j)+1)}| S_{i+1}=S_i \cup
v_{i+1}]), 
f ( E[\sum_{v \in H} \frac{1}{d_H(v_j)+1)}| S_{i+1}=S_i] ).
$$
Let $\psi_{i+1}^{+}$ denote the value of $\psi_{i+1}$ with
$S_{i+1}=S_i \cup v_{i+1}$ and $\psi_{i+1}^{-}$
denote the value of $\psi_{i+1}$ with $S_{i+1}=S_i$.

By adding the last inequality and (\ref{e713}),(\ref{e7131})  
we conclude that
$$
\psi_i \geq \min \{\psi_{i+1}^+, \psi_{i+1}^{-} \}.
$$
Therefore, if the algorithm decides in each step $i+1$ whether
or not to  add $v_{i+1}$ to $S_i$ in order to get $S_{i+1}$ 
by choosing the option that minimizes the value of
$\psi_{i+1}$, then the potential function
$\psi_i$ is a monotone decreasing function of $i$.
Since $\psi_0$ is at most 
$\frac{n}{k+1}( \ln (k+1)+4)-2$ by the proof of Theorem
\ref{t714}, so is $\psi_n$. However,
$\psi_n$ is exactly $|S|+f(D(H))$  for the dominating set $S$
constructed by the algorithm. This completes the description
of the algorithm and its correctness.

\subsection{Problem}
We conclude with the following problem.
\vspace{0.1cm}

\noindent
{\bf Problem:}\, Determine or estimate the maximum possible value
of the difference $\gamma_c(G)-\gamma(G)$, where the maximum is
taken over all connected graphs $G$ with $n$ vertices and minimum
degree at least $k$. 
\vspace{0.2cm}

\noindent
By Theorem \ref{t713} this maximum is at most $\frac{n}{k+1}(\ln
\lceil \ln (k+1) \rceil +3)$. It is not difficult to show that it
is at least  $\lfloor \frac{n}{k+1} \rfloor-1$. 
To see this assume, for simplicity, that $k+1$ divides $n$ and 
put $m=\frac{n}{k+1}$.
For each $0 \leq i <m$ let $K_i$ be the graph
obtained from a clique on $k+1$ vertices by deleting a single edge
$x_iy_i$. Let $G$ be the $k$-regular graph obtained from the vertex
disjoint union of the $m$ graphs $K_i$ by adding the edges
$y_ix_{i+1}$ for all $0 \leq i <m$, where $x_m=x_0$. For this cycle
of cliques $G$, $\gamma(G)=m=\frac{n}{k+1}$ as shown by a dominating
set consisting of one vertex in each $K_i-\{x_i,y_i\}$ - this is a 
minimum dominating set as $G$ is $k$-regular. 
On the other hand the induced subgraph on
any connected dominating set must contain at least $m-1$ of the
edges $y_ix_{i+1}$ and their endpoints, and it is not difficult to
check that it must contain at least one additional vertex. Thus
$\gamma_c(G)=2m-1=2 \frac{n}{k+1}-1$. It will be
interesting to close the $\ln \ln (k+1)$ gap between the upper and
lower bounds and decide whether or not the above maximum is 
$\Theta(\frac{n}{k+1})$.

\noindent
{\bf Acknowledgment}
I thank Eli Berger, Michael Krivelevich, Shay Moran, 
Shubhangi Saraf, Madhu Sudan and Tibor Szab\'o for helpful
discussions.

\end{document}